\documentclass[12pt]{article}

\usepackage{amsmath,amsthm,amsfonts,amssymb, color}

\newtheorem{theorem}{Theorem}[section]

\newtheorem{lemma}[theorem]{Lemma}

\newtheorem{remark}[theorem]{Remark}

\def\cB{\mathcal{B}}

\def\cF{\mathcal{F}}

\def\cP{\mathcal{P}}

\def\cS{\mathcal{S}}

\def\bE{\mathbb{E}}
\def\bN{\mathbb{N}}
\def\bP{\mathbb{P}}
\def\bQ{\mathbb{Q}}
\def\bR{\mathbb{R}}

\def\e{\varepsilon}

\begin{document}

\title{Stochastic wave equation with L\'evy white noise}

\author{Raluca M. Balan\footnote{University of Ottawa, Department of Mathematics and Statistics, 150 Louis Pasteur Private, Ottawa, Ontario, K1G 0P8, Canada. E-mail address: rbalan@uottawa.ca.}\footnote{Research supported by a grant from the Natural Sciences and Engineering Research Council of Canada.}}

\date{March 22, 2023}
\maketitle

\begin{abstract}
\noindent In this article, we study the stochastic wave equation on the entire space $\bR^d$, driven by a space-time L\'evy white noise with possibly infinite variance (such as the $\alpha$-stable L\'evy noise). In this equation, the noise is multiplied by a Lipschitz function $\sigma(u)$ of the solution. We assume that the spatial dimension is $d=1$ or $d=2$. Under general conditions on the L\'evy measure of the noise, we prove the existence of the solution, and we show that, as a function-valued process, the solution has a c\`adl\`ag modification in the local fractional Sobolev space of order $r<1/4$ if $d=1$, respectively $r<-1$ if $d=2$.
\end{abstract}

\noindent {\em MSC 2020:} Primary 60H15; Secondary 60G60, 60G51

\vspace{1mm}

\noindent {\em Keywords:} stochastic partial differential equations, random fields, space-time L\'evy white noise

\section{Introduction}

The study of stochastic partial differential equation (SPDEs) using the random field approach was initiated by John Walsh's lecture notes \cite{walsh86} in 1986, and has become a very broad area in stochastic analysis since then. Major efforts have been dedicated to understanding the behaviour of solutions of SPDEs driven by  space-time Gaussian white noise, as a natural replacement for the Brownian motion that is used in the classical theory of stochastic differential equations (SDEs). In 1999, in the seminal article \cite{dalang99}, Robert Dalang introduced the spatially-homogeneous Gaussian noise and developed the major tools for the study of SPDEs using random fields, laying the foundation of a general theory. These tools were embraced very quickly by a large scientific community and yielded spectacular results, especially when combined with techniques from Malliavin calculus. We include here just a small sample from a very large set of important contributions to this area: \cite{conus-dalang08,hu-nualart09,FK09,CJK13,chen-dalang15,
huang-le-nualart15,xchen17}.

We should also mention that there exists an alternative way for studying SPDEs, using cylindrical processes as noise, and there is a vast literature dedicated to stochastic evolution equations with this type of noise.
We refer the reader to the excellent monograph \cite{PZ07} and the references therein. The recent papers \cite{jakubowski-riedle17,kosmala-riedle21} contain significant advances on stochastic integration with respect to cylindrical L\'evy noise, while \cite{griffiths-riedle21} gives a comparison between the cylindrical approach and the random field approach, which complements the similar comparison that was done in \cite{dalang-quer11} in the Gaussian case.

SDEs driven by L\'evy processes have been in the literature as long as their Brownian motion counterparts, as their origin can be traced back to It\^o's memoir \cite{ito51}.
There exist many breakthrough contributions to the area of SDEs driven by L\'evy processes (or more generally, by discontinuous semi-martingales), and several excellent monographs were published on this subject (for instance \cite{applebaum09}, \cite{protter}, \cite{bichteler02}, \cite{JS}).
But the study of SPDEs driven by L\'evy noise using the random field approach is not so well-developed as
 the Gaussian case. The majority of the existing works have focused so far only on the stochastic heat equation.
We recall below some known results for the heat equation, which are relevant for the present article.
But first, we need to introduce some notation.

Throughout this article, we let $L=\{L(A);A \in \cB_0(\bR_{+} \times \bR^d)\}$ be a {\em pure-jump L\'evy space-time white noise}, defined on a complete probability space $(\Omega,\cF,\bP)$:
\begin{equation}
\label{def-L}
L(A)=b|A|+\int_{A \times \{|z|\leq 1\} } z \widetilde{J}(dt,dx,dz)+
\int_{A \times \{|z|> 1\}}z J(dt,dx,dz),
\end{equation}
where $\cB_0(\bR_{+} \times \bR^d)$ is the class of Borel sets in $\bR_{+} \times \bR^d$ with finite Lebesgue measure, $|A|$ is the Lebesgue measure of $A$, $b \in \bR$, $J$ is a Poisson random measure on $\bR_{+} \times \bR^d \times \bR$ of intensity $dtdx \nu(dz)$, $\widetilde{J}$ is the compensated version of $J$,
and $\nu$ is a L\'evy measure on $\bR$, i.e. a measure satisfying the following conditions:
\begin{equation}
\label{cond-Levy}
\int_{\bR}(|z|^2 \wedge 1)\nu(dz)<\infty \quad \mbox{and} \quad \nu(\{0\})=0.
\end{equation}
An important particular case is when $L$ is an {\em $\alpha$-stable L\'evy noise}, i.e. an $\alpha$-stable random measure with control measure given by the Lebesgue measure multiplied by a constant, as defined in \cite{ST94}. In this case,
\[
\nu(dz)=\Big(c_{+} z^{-\alpha-1}1_{(0,\infty)}(z)+c_{-} (-z)^{-\alpha-1}1_{(-\infty,0)}(z)\Big)dz
\]
for some $\alpha \in (0,2)$, $c_{+}>0$, $c_{-}>0$. For the symmetric $\alpha$-stable L\'evy noise, $c_{+}=c_{-}$.

The noise $L$ is a natural space-time extension of a classical L\'evy process (with no Gaussian component), which we recall can be written as
\[
X(t)=at+\int_{[0,t] \times \{|z|\leq 1\}}z \widetilde{N}(dt,dz)+  \int_{[0,t] \times \{|z|> 1\}}z N(dt,dz), \quad t\geq 0,
\]
where $a \in \bR$, $N$ is a Poisson random measure on $\bR_{+} \times \bR$ of intensity $dt \nu(dz)$, and $\nu$ is a L\'evy measure. The process $\{X(t)\}_{t\geq 0}$ has a c\`adl\`ag modification, and the measure $\nu$ gives information about the size of the jumps of the sample paths of this modification.

In the space-time framework, we still speak about the third component of $J$ as the ``jump'' component, although we do not identify a time-indexed process whose paths have these ``jumps''.
 In particular, the ``large jumps'' (corresponding to $|z|>1$) control the moments of $L$, in the sense that $\bE|L(A)|^p<\infty$ if $\int_{|z|>1}|z|^p\nu(dz)<\infty$, for any $p\geq 2$.
In the present paper, we are interested in the case when
$\int_{|z|>1}|z|^2\nu(dz)$ may be infinite, and $L(A)$ may have infinite variance, as it happens in the $\alpha$-stable case. (The case of finite variance L\'evy noise is interesting too, since one can develop Malliavin calculus techniques similar to the Gaussian case; see \cite{NN, DOP09} for a very readable introduction to this subject.)

\medskip
Consider now the following stochastic heat equation with noise $L$:
\begin{align}
\label{heat-eq} 
\begin{cases}
\dfrac{\partial u}{\partial t} (t,x)=\frac{1}{2}\Delta u(t,x)+\sigma(u(t,x))\dot{L}(t,x), \quad \quad t>0,x \in \bR^d \quad (d\geq 1),\\
u(0,x) = u_0(x), \ x \in \bR^d.
\end{cases}
\end{align}
The solution of this equation is a predictable process which satisfies the integral equation:
\[
u(t,x)=(g_t*u_0)(x)+\int_0^t \int_{\bR^d}g_{t-s}(x-y)\sigma(u(s,y))L(ds,dy),
\]
where $g_t(x)=(2\pi t)^{-d/2}\exp(-|x|^2/(2t))$ is the heat kernel, and $|\cdot|$ is the Euclidean norm in $\bR^d$. We say that a random field $\{X(t,x);t\geq 0,x\in \bR^d\}$ is {\em predictable} if it is measurable with respect to the $\sigma$-field $\cP=\cP_0 \times \cB(\bR^d)$, where $\cP_0$ is the predictable $\sigma$-field on $\Omega \times \bR_{+}$.

In \cite{bie98}, Saint Loubert Bi\'e proved that if the measure $\nu$ satisfies
\begin{equation}
\label{bie-cond}
\int_{\bR}|z|^p \nu(dz)<\infty \quad \mbox{for some} \quad p\in[1,2],
\end{equation}
and
\begin{equation}
\label{cond-p}
p<1+\frac{2}{d}
\end{equation}
then equation \eqref{heat-eq} has a unique solution which satisfies:
\[
\sup_{(t,x)\in [0,T] \times \bR^d}\bE|u(t,x)|^p<\infty.
\]

Condition \eqref{cond-p} comes from the requirement $\int_0^t \int_{\bR^d}g_{t-s}^p(x-y)dyds<\infty$ and forces $p<2$ if $d\geq 2$. A more severe restriction is \eqref{bie-cond}, since it excludes the $\alpha$-stable L\'evy noise. This has been an open problem for some time, which was partially solved in \cite{B14,chong17-JTP} by replacing $\bR^d$ in \eqref{heat-eq} by a bounded domain $D \subset \bR^d$. Previous investigations related to this problem can be found for instance in \cite{AWZ98,applebaum-wu00,mueller98,mytnik02}, the last two references dealing with $\alpha$-stable L\'evy noise and a non-Lipschitz function $\sigma$.

A major breakthrough was made in article \cite{chong17-SPA}, in which it was showed that equation \eqref{heat-eq} has a solution, if there exist some exponents $p>0,q>0$ such that $p$ satisfies \eqref{cond-p}, $p/(2+2/d-p)<q\leq p$, and
\[
\int_{\{|z|\leq 1\}}|z|^p \nu(dz)<\infty \quad \mbox{and} \quad
\int_{\{|z|>1\}} |z|^{q}\nu(dz)<\infty.
\]
This condition holds for the $\alpha$-stable L\'evy noise with $\alpha<1+2/d$. If $p<1$, it is assumed in addition that $b=\int_{\{|z|\leq 1\}}z\nu(dz)$, a condition which holds for the symmetric $\alpha$-stable L\'evy noise. The novel ideas of \cite{chong17-SPA} are to use different exponents $p$ and $q$ for $|z|\leq 1$ and $|z|>1$, and a spatially-dependent truncation function $h(x)$.
The constant truncation function $h(x)=1$ used in \cite{B14} for the $\alpha$-stable L\'evy noise is problematic for equations on the entire domain $\bR^d$.
Unlike the Gaussian case, the solution of \cite{chong17-SPA} is not obtained as the limit of the sequence of Picard's iterations, being defined as
$u(t,x)=u_{N}(t,x)$ if $t\leq \tau_N$,
where $u_N$ is the solution of equation \eqref{heat-eq} in which $L$ is replaced by the truncated noise:
\begin{align}
\nonumber
L_N(A)&=b|A|+\int_{A \times \{|z|\leq 1\}}z \widetilde{J}(dt,dx,dz)+
\int_{A \times \{1<|z|\leq N h(x) \}}z J(dt,dx,dz)\\
\label{decomp-LN}
&=:b|A|+L^M(A)+L_N^P(A)
\end{align}
where the indices $M$ and $P$ come from ``martingale'', respectively ``compound-Poisson''. The stopping times $\tau_N$ are given by:
\[
\tau_N=\inf \left\{T>0; \int_0^T \int_{\bR^d} \int_{\{|z|>N h(x)\}} J(dt,dx,dz)>0\right\},
\]
where $h(x)=1+|x|^{\eta}$ for some $\frac{d}{q}<\eta<\frac{2-d(p-1)}{p-q}$.
Since $u_N(t,x)=u_{N+1}(t,x)$ a.s. on the event $\{t \leq \tau_{N}\}$, $u(t,x)$ is well-defined.
In addition, it was shown in \cite{chong17-SPA} that $\bE|u_N(t,x)|^p$ is finite and uniformly bounded for $(t,x)$ in a compact set. But the solution may not be unique.
The asymptotic behaviour of the moments of $u(t,x)$ have been studied in the subsequent paper \cite{chong-kevei19}. The path properties of the solution have been studied in \cite{CDH19}. The
recent preprint \cite{BCL} focuses on the case $\sigma(u)=\beta u$, $\beta>0$ and $\nu(-\infty,0)=0$, and establishes uniqueness of the solution, together with deep intermittency-type properties.

\medskip

In this article, we analyze the stochastic wave equation:
\begin{align}
\label{wave-eq} 
\begin{cases}
\dfrac{\partial^2 u}{\partial t^2} (t,x)=\Delta u(t,x)+\sigma(u(t,x))\dot{L}(t,x), \quad \quad t>0,x \in \bR^d \quad (d \leq 2), \\
u(0,x) = u_0(x), \quad \dfrac{\partial u}{\partial t}(0,x)=v_0(x), \quad \quad \quad \quad  \ x \in \bR^d,
\end{cases}
\end{align}
where $\sigma$ is a Lipschitz function on $\bR$, $u_0$ and $v_0$ are deterministic functions on $\bR^d$, and $L$ is a L\'evy space-time white noise given by \eqref{def-L}. The reason we restrict ourselves to the case $d\leq 2$ is that in dimensions $d\geq 3$, the fundamental solution $G_t$ of the wave equation is not a function (it is a measure if $d=3$ and a distribution if $d\geq 4$).

A predictable random field $u=\{u(t,x);t\geq 0,x\in \bR^d\}$ is a {\bf solution} of \eqref{wave-eq} if it satisfies the integral equation:
\[
u(t,x)=w(t,x)+\int_0^t \int_{\bR^d}G_{t-s}(x-y)\sigma(u(s,y))L(ds,dy).
\]
The stochastic integral on the right-hand side of this equation is defined as in \cite{chong17-SPA}, using the concept of {\em Daniell mean}. We refer the reader to the appendix for the definition of this integral.

Here $G$ is the fundamental solution of the wave equation on $\bR_{+}\times \bR^d$: for any $t>0$ 
\begin{align}
	G_t(x)=
	\begin{cases}
		\displaystyle \frac{1}{2}1_{\{|x|<t\}}                               & \text{if $d=1$}\\[1em]
		\displaystyle \frac{1}{2\pi} \frac{1}{\sqrt{t^2-|x|^2}}1_{\{|x|<t\}} & \text{if $d=2$}
	\end{cases}
\end{align}
and $w$ is the solution of the wave equation
$\frac{\partial u}{\partial t}-\Delta u=0$ on $\bR_{+} \times \bR^d$ with the same in initial conditions as \eqref{wave-eq}, given by:
\[
w(t,x)=(G_t*v_0)(x)+\frac{\partial}{\partial t} (G_t*u_0)(x).
\]
We assume that the functions $u_0$ and $v_0$ satisfy the following conditions:
\begin{itemize}
\item ($d=1$) $u_0$ is bounded and continuous, and $v_0$ is bounded and measurable
\item ($d=2$) $u_0 \in C^1(\bR^2)$ and there exists $q_0 \in (2,\infty]$ such that $u_0,\nabla u_0,v_0 \in L^{q_0}(\bR^2)$.
\end{itemize}
Under these conditions, $w$ is jointly continuous and $\sup_{(t,x)\in [0,T] \times \bR^d}|w(t,x)|<\infty$ (see for instance, Lemma 4.2 of \cite{dalang-quer11}).

Our first goal is to prove that the solution to equation \eqref{wave-eq} exists.
As far as we know, this problem has not been studied in the literature before. In dimension $d=2$, one difficulty is the fact that the fundamental solution $G_t(x)$ of the wave operator has singularites on the boundary of the set $|x|<t$, which lead to lengthy calculations (see for instance \cite{dalang-frangos98,millet-sanz99} for the case of Gaussian noise). Another problem is the fact that $(G_t)_{t> 0}$ does not have the semigroup property.
Luckily, some extremely useful (and highly non-trivial) properties of $G$ have been recently obtained in article \cite{BNZ20}, which lead to very impressive results about the asymptotic behaviour of the spatial average of the solution of the wave equation with spatially-homogeneous Gaussian noise. These properties will play an important role in this article, for the proof of the existence of the solution.
More precisely, the proof of Theorem \ref{exist-th-K} below involves working with convolutions of the form $G_t^p*G_s^p$; see relation \eqref{def-B-in} below. In the case of the heat equation, these are dealt with in \cite{chong17-SPA} using properties of the normal distribution, which reduce essentially to the semigroup property of the heat kernel $g_t(x)$, since $g_t^p(x)=(2\pi t)^{d(1-p)/2}g_t(x)$. In the case of the wave equation, it is not obvious how to work with such convolutions, especially in the case $d=2$. If $d=1$, the problem is not so difficult since $(G_t*G_s)/(ts)$ is the law of the sum of two independent random variables with uniform distributions on $(-t,t)$, respectively $(-s,s)$.

Our second goal is to show that the solution of equation \eqref{wave-eq} has a c\`adl\`ag modification, when viewed as a process with values in a suitable fractional Sobolev space. A similar phenomenon has been studied in \cite{CDH19} for the heat equation. The starting point of this analysis is a quick look at the behavior of $G_t$ when $t=0$. In the case of the heat equation,
$g_0(x):=\lim_{t\to 0}g_t(x)$ is 0 if $x \not=0$ and $\infty$ if $x=0$, in any dimension $d$; hence, we can say that $g_0=\delta_0$ (the Dirac delta distribution at $0$). In the case of the wave equation, the situation is different in the case $d=2$ compared with $d=1$: $G_0=\delta_0$ if $d=2$, and $G_0=2^{-1}1_{\{0\}}$ if $d=1$ ($1_{\{0\}}$ is indicator function of $\{0\}$). Since the atoms $(T_i,X_i,Z_i)$ of $J$ contribute a value $G_{t-T_i}(x-X_i)Z_i$ to the solution $u(t,x)$, the function $u(t,\cdot)$ will likely live in the same function space as $G_{t-T_i}(\cdot-X_i)$. Moreover, the regularity of the path $t\mapsto u(t,\cdot)$ is related to regularity of $t \mapsto G_{t-T_i}(\cdot-X_i)$. In Section 3 below, we will show that the solution to equation \eqref{wave-eq} has a c\`adl\`ag modification with values in the fractional Sobolev space $H_{\rm loc}^r(\bR^d)$, for any $r<-1$ if $d=2$, respectively for any $r<1/4$ if $d=1$.

The article is organized in two parts: in Section 2 we show the existence of the solution to equation \eqref{wave-eq}, and in Section 3 we study the path properties of this solution. We recall that the $L^p(\Omega)$-norm is defined by:
\[
\|X\|_p=(\bE|X|^p)^{1/p} \quad \mbox{if $p\geq 1$} \quad \mbox{and} \quad
\|X\|_p=\bE|X|^p \quad \mbox{if $p<1$}.
\]

\section{Existence of solution}

In this section, we prove the existence of solution to equation \eqref{wave-eq}.
We proceed as in \cite{chong17-SPA} in the case of the heat equation. Let $L_N$ be the truncated noise given by \eqref{decomp-LN}, with
$h(x)=1+|x|^{\eta}$ for some $\eta>0$.
We assume that $\nu$ satisfies the following assumption:

\medskip

{\bf Assumption A.} There exist $0<q \leq p \leq 2$ such that
\begin{equation}
\label{pq-cond}
\gamma_1=\int_{\{|z|\leq 1\}}|z|^{p}\nu(dz)<\infty \quad \mbox{and} \quad \gamma_2=\int_{\{|z|>1\}}|z|^q \nu(dz)<\infty.
\end{equation}
If $p<1$, assume that $b=\int_{\{|z|\leq 1\}}z\nu(dz)$.

\begin{remark}
{\rm
{\em (i)}
In addition to \eqref{pq-cond}, we will need that $\int_0^t \int_{\bR^d}G_{t-s}^p(x-y)dyds<\infty$. If $d=1$, this imposes no restrictions on $p$, so we can take $p=2$. But if $d=2$, we encounter the restriction $p<2$, which is the same as condition \eqref{cond-p} that is needed for the heat equation in dimension $d=2$.

{\em (ii)} We will see below that the value $p$ from Assumption A plays an important role in the analysis of the solution $u_N$ of the wave equation with truncated noise $L_N$: $u_N(t,x)$ has a finite $p$-th moment! Since in dimension $d=1$, we can take $p=2$, this means that $u_N(t,x)$ has finite second moment. 

{\em (iii)} If $L$ is an $\alpha$-stable L\'evy noise, condition \eqref{pq-cond} holds for any $q<\alpha<p$. So in dimension $d=1$, the solution of the equation with truncated noise has finite second moment, 
although the noise itself does not have this property. This shows the drastic impact of the truncation of the noise on the behaviour of the solution.

{\em (iv)} Unlike the case of the heat equation, we do not require a lower bound on the exponent $q$ in \eqref{pq-cond}. In \cite{chong17-SPA}, the fact that the upper bound given by (3.13) has to be summable in $n$, combined with the restriction $\eta>d/q$ (from Lemma 3.2 ibid.) yields the condition $q>\frac{dp}{d+2-d(p-1)}$. For the wave equation, we obtain a different upper bound, which is summable regardless of the value of $q$ (see the proof of Theorem \ref{exist-th-K} below).

{\em (v)} In the recent preprint \cite{BCL}, the authors have obtained the existence and uniqueness of solution to the heat equation \eqref{heat-eq} with $\sigma(u)=u$ (known as the parabolic Anderson model) under different integrability conditions on the small jumps and the large jumps than \eqref{pq-cond}. The methods of \cite{BCL} rely on the special form of the heat kernel. Finding similar methods that would show the existence and uniqueness of solution to the wave equation \eqref{wave-eq} with $\sigma(u)=u$ (known as the hyperbolic Anderson model) remains an open problem.
}
\end{remark}

When Assumption A holds with $p<1$, the truncated noise $L^N$ has no drift:
\begin{align}
\nonumber
L_N(A)&=\int_{A \times \{|z|\leq 1\}}z J(dt,dx,dz)+\int_{A \times\{1<|z|\leq N h(x)\}}z J(dt,dx,dz)\\
\label{no-drift-LN}
&=:L^Q(A)+L_N^P(A).
\end{align}

\medskip

The following lemma is an important tool for controlling the $p$-th moments of stochastic convolutions of $G$ with the truncated noise $L_N$, and is a reformulation of Lemma 3.3 of \cite{chong17-SPA}. We believe that the condition $p\leq 2$ is needed for this result, since its proof relies essentially on the maximal inequality given by Theorem \ref{max-ineq-th}. This condition is missing in \cite{chong17-SPA}.
Note that this result was stated in \cite{chong17-SPA} for the fundamental solution of the heat equation, but it
remains valid for general functions. 

\begin{lemma}[Lemma 3.3 of \cite{chong17-SPA}]
\label{lem-p-mom}
Suppose that Assumption A holds. Let $G$ be a non-negative function such that $
\int_0^T \int_{\bR^d}\big(G_{t}^p(x)+G_{t}(x)\big)dxdt<\infty$.
For any predictable processes $X,X_1$ and $X_2$ and for any $t \in [0,T]$ and $x \in \bR^d$,
\begin{align*}
& \bE\left|\int_0^t \int_{\bR^d} G_{t-s}(x-y) X(s,y)L_N(ds,dy)\right|^p \leq \\
& \quad
 C_{T} \int_0^t \int_{\bR^d} \big(G_{t-s}^p(x-y)+G_{t-s}(x-y)1_{\{p \geq 1\}}\big) \big(1+\bE|X(s,y)|^p \big) h(y)^{p-q}dyds
\end{align*}
and
\begin{align*}
& \bE\left|\int_0^t \int_{\bR^d} G_{t-s}(x-y) X_1(s,y)L_N(ds,dy)-\int_0^t \int_{\bR^d} G_{t-s}(x-y) X_2(s,y)L_N(ds,dy)\right|^p \leq \\
& C_{T} \int_0^t \int_{\bR^d} \big(G_{t-s}^p(x-y)+G_{t-s}(x-y)1_{\{p \geq 1\}}\big) \bE|X_1(s,y)-X_2(s,y)|^p  h(y)^{p-q}dyds,
\end{align*}
where $C_{T}$ is a constant that depends on $T$ (and also on $K,p, q,\gamma_1,\gamma_2$, but not on $h$), and $p,q$ are the constants from Assumption A.
\end{lemma}

We will need several properties of $G$.
In both cases $d=1$ and $d=2$, $\int_{\bR^d}G_t(x)dx=t$, and the Fourier transform of $G_t$ is:
\begin{equation}
\label{Fourier-G}
\cF G_t(\xi)=\int_{\bR^d}e^{-i\xi \cdot x}G_t(x)dx=\frac{\sin(t|\xi|)}{|\xi|}, \quad \mbox{for all $\xi \in \bR^d,t>0$}.
\end{equation}

Note that
\begin{align}
\label{p-G}
	\int_{\bR^d}G_t^p(x)dx=
	\begin{cases}
		\displaystyle    2^{1-p}t                          & \text{for any $p>0$, \ if $d=1$}\\[1em]
		\displaystyle  \frac{(2\pi)^{1-p}}{2-p}t^{2-p} & \text{for any $p \in (0,2)$, \ if $d=2$}\\
	\end{cases}
\end{align}

If $d=1$, $\int_{\bR} G^p(t,x)|x|^{\gamma}dx=\frac{2^{1-p}}{\gamma+1}t^{\gamma+1}$ for any $p>0$ and $\gamma>-1$.
If $d=2$, 
\begin{equation}
\label{Gp-gamma}
\int_{\bR^2} G^p(t,x)|x|^{\gamma}dx \leq \frac{(2\pi)^{1-p}}{2-p}t^{2-p+\gamma} \quad \mbox{for any $p\in (0,2)$ and $\gamma>0$,}
\end{equation}
and
\begin{equation}
\label{G-pq-ineq}
G_t^{p}(x) \leq (2\pi t)^{q-p} G_t^{q}(x) \quad \mbox{for any $p<q$}.
\end{equation}

In the case of the wave equation, $(G_t)_{t\geq 0}$ does not have the semigroup property. Fortunately, the recent article \cite{BNZ20} contains some very useful properties of $G$, from which one can deduce a sub-semigroup-type property of $(G_t)_{t\geq 0}$. When $d=1$, it is not difficult to see that for any
$r<s<t$ and $x \in \bR$,
\begin{equation}
\label{semigroup-1}
(G_{t-s}*G_{s-r})(x) \leq \frac{1}{2} (t-r)G_{t-r}(x).
\end{equation}
This is relation (2.6) of \cite{NZ21}. For $d=2$, we have the following highly non-trivial result.

\begin{lemma}[Lemma 4.3 of \cite{BNZ20}]
\label{lem-BNZ}
Assume that $d=2$. Let $q \in (\frac{1}{2},1)$ and $\delta \in [1,1/q]$ be arbitrary. Then, for any $0<r<t$ and $x \in \bR^2$,
\begin{equation}
\label{lem33-BNZ}
\int_r^t (G_{t-s}^{2q}*G_{s-r}^{2q})^{\delta}(x)ds \leq A_{q} (t-r)^{1-\delta(2q-1)}G_{t-r}^{\delta(2q-1)}(x),
\end{equation}
where $A_{q}>0$ is a constant depending on $q$.
\end{lemma}

Using \eqref{lem33-BNZ} and H\"older's inequality, we obtain that for any $q \in (\frac{1}{2},1)$ and $p<2q$,
\begin{equation}
\label{I-rt-z}
\int_{r}^{t} (t-s)^{2q-p} (s-r)^{2q-p} (G_{t-s}^{2q}* G_{s-r}^{2q})(x)ds \leq C_{p,q} (t-r)^{2(q-p+1)} G_{t-r}^{2q-1}(x),
\end{equation}
where $C_{p,q}>0$ is a constant depending on $p$ and $q$.

\medskip

We will need the following result, which existed in the first version of \cite{BNZ20} posted on arXiv (Lemma 4.3 in preprint arXiv:2003.10346v1), but was removed from the final version of \cite{BNZ20}. The authors have confirmed in personal communication that this result is correct. We list it here with a reference to \cite{NZ21}, where it was used too.

\begin{lemma}[Lemma 2.5.(b) of \cite{NZ21}]
\label{lem-BNZ2}
Assume that $d=2$. Let $q \in (\frac{1}{2},1)$ and $p \in (0,1)$ be such that $p+2q \leq 3$.
For any $0<r<t$ and $x \in \bR^2$,
\[
\int_r^t (G_{t-s}^{2q}*G_{s-r}^{p})(x)ds \leq B_{p,q} (t-r)^{3-p-2q} 1_{\{|x|<t-r\}},
\]
where $B_{p,q}>0$ is a constant depending on $p$ and $q$.
\end{lemma}

\medskip

We will need also the following elementary result about a  multiple beta-type integral.

\begin{lemma}
\label{lem-beta}
For any $\beta_1>-1,\ldots,\beta_n>-1$,
\[
\int_{T_n(t)}\prod_{j=1}^{n}(t_{j+1}-t_j)^{\beta_j}dt_1 \ldots dt_n=
\frac{\prod_{j=1}^{n}\Gamma(\beta_j+1)}{\Gamma(\sum_{j=1}^{n}\beta_j+n+1)} t^{\sum_{j=1}^n\beta_j +n},
\]
where $T_n(t)=\{(t_1,\ldots,t_n) \in (0,t)^n;t_1<\ldots<t_n\}$ and $t_{n+1}=t$.
\end{lemma}

We consider the equation with truncated noise $L_N$:
\begin{align}
\label{wave} 
\begin{cases}
\dfrac{\partial^2 u_N}{\partial t^2} (t,x)=\Delta u_N(t,x)+\sigma(u_N(t,x))\dot{L}_{N}(t,x) \quad \quad t>0,x \in \bR^d \quad (d \leq 2) \\
u_N(0,x) = u_0(x), \quad \dfrac{\partial u_N}{\partial t}(0,x)=v_0(x) \quad \quad \quad \quad \qquad \ x \in \bR^d
\end{cases}
\end{align}

A predictable process $u_N$ is a {\bf solution} of \eqref{wave} if it satisfies the integral equation:
\begin{equation}
\label{eq-uN}
u_{N}(t,x)=w(t,x)+\int_0^t \int_{\bR^d}G_{t-s}(x-y)\sigma(u_{N}(s,y)) L_N(ds,dy).
\end{equation}

We say that two random fields $\{X(t,x);t\geq 0,x\in \bR^d\}$ and $\{Y(t,x);t\geq 0,x \in \bR^d\}$ are {\em modifications} of each other if $\bP(X(t,x)=Y(t,x))=1$ for almost all $(t,x) \in \bR_{+} \times \bR^d$.

\medskip

We are now ready to prove the existence of solution for the equation with truncated noise. Note that unlike the case of the heat equation considered in \cite{chong17-SPA}, we do not need to impose any additional conditions on $p,q$ and $\eta$, for a {\em fixed} truncation level $N$. The additional condition $\eta>d/q$ which is imposed in the proof of Theorem \ref{exist-th} below guarantees that we can paste together the solutions $u_N$ for different truncation levels $N$, to produce a solution for the equation with non-truncated noise.

\begin{theorem}
\label{exist-th-K}
(i) If $d=1$, suppose that there exists $0<q\leq 2$ such that $\int_{|z|>1}|z|^q\nu(dz)<\infty$. Then equation \eqref{wave}
has a 
solution $u_{N}=\{u_{N}(t,x);t\geq 0,x\in \bR\}$. Moreover,
\begin{equation}
\label{p-mom-finite-2}
\sup_{t \in [0,T]} \sup_{|x|\leq R}\bE |u_{N}(t,x)|^{2}<\infty,
\end{equation}
for any $T>0$ and $R>0$.

(ii) If $d=2$, suppose that Assumption A holds with $p<2$. Then equation \eqref{wave}
has a unique solution $u_{N}=\{u_{N}(t,x);t\geq 0,x\in \bR^d\}$. Moreover,
\begin{equation}
\label{p-mom-finite}
\sup_{t \in [0,T]} \sup_{|x|\leq R}\bE |u_{N}(t,x)|^{p}<\infty,
\end{equation}
for any $T>0$ and $R>0$, where $p$ is the constant from Assumption A.
\end{theorem}

\begin{proof} To provide a unified argument for both cases, we let $p=2$ when $d=1$.
We consider the sequence
$\{u_N^{(n)}(t,x)\}_{n\geq 0}$ of Picard's iterations (specific to the truncation level $N$), defined by $u_{N}^{(0)}(t,x)=w(t,x)$,
\begin{equation}
\label{picard}
u_{N}^{(n+1)}(t,x)=w(t,x)+\int_0^t \int_{\bR^d}G_{t-s}(x-y)\sigma(u_{N}^{(n)}(s,y)) L_N(ds,dy) \quad n\geq 0.
\end{equation}
By Lemma 6.2 of \cite{chong17-JTP}, $u_N^{(n)}$ has a predictable modification. We work with this modification when defining $u_{N}^{(n+1)}$.
By induction on $n$, it can be proved that for any $t>0$ and $x \in \bR^d$,
\[
\bE|u_{N}^{(n)}(t,x)|^p \leq C_{n,t}(1+|x|^{n\eta(p-q)}),
\]
where $C_{n,t}$ is a positive constant that depends on $n,t$ and is increasing in $t$. For this, we use Lemma \ref{lem-p-mom} and the fact that for any $\gamma>0$,
\[
\int_{0}^t \int_{\bR^d} G_{t-s}^p(x-y)(1+|y|^{\gamma})dyds \leq C_{\gamma,p,t} (1+|x|^{\gamma})
\]
where $C_{\gamma,p,t}$ is a positive constant that depends on $\gamma,p,t$ and is increasing in $t$. Hence, for any $t>0$ and $x \in \bR^d$, $u_N^{(n)}(t,x)$ is finite a.s.

 It suffices to prove that:
\begin{equation}
\label{aim}
\sum_{n\geq 1} \sup_{t \in [0,T]} \sup_{|x|\leq R} \|u_N^{(n)}(t,x)-u_{N}^{(n-1)}(t,x)\|_p<\infty.
\end{equation}
This will imply that $\{u_N^{(n)}(t,x)\}_{n\geq 0}$ is a Cauchy sequence in $L^p(\Omega)$, uniformly in $(t,x) \in [0,T] \times \{x \in \bR^d;|x|\leq R\}$. We denote its limit by $u_N(t,x)$, that is:
\begin{equation}
\label{unif-conv}
\sup_{(t,x) \in [0,T]}\sup_{|x|\leq R}\bE|u_N^{(n)}(t,x)-u_{N}(t,x)|^p \to 0 \quad \mbox{as} \quad n \to \infty.
\end{equation}

Once \eqref{aim} is shown, the rest of the proof will be the same as for the heat equation (Theorem 3.1 of \cite{chong17-SPA}). But there is a delicate part regarding the existence of a predictable modification of $u_N$. We include this argument here since it is not given in \cite{chong17-SPA}.

We fix $(t,x)$ and let $V_{N}^{(n)}(s,y)=G_{t-s}(x-y) u_N^{(n)}(s,y) h(y)^{\frac{p-q}{p}}$ for any $(s,y) \in (0,t) \times \bR^d$. Then $V_{N}^{(n)}$ is predictable, since $u_N^{(n)}$ is so. The proof of \eqref{aim} given below shows that
\[
\sum_{n\geq 1}\|V_{N}^{(n)}-V_{N}^{(n-1)}\|_{L^p(\Omega \times (0,t) \times \bR^d)}<\infty.
\]
Hence, $\{V_{N}^{(n)}\}_{n\geq 0}$ is a Cauchy sequence in
$L^p(\Omega \times (0,t) \times \bR^d)$. We denote by $V_{N}$ its limit in this space. So, there exists a subsequence $N' \subset \bN$ such that $V_{N}^{(n)}(\omega,s,y) \to V_{N}(\omega,s,y)$ as $n\to \infty$, for almost all $(\omega,s,y)$, i.e. for $(\omega,s,y) \in B^c$, for some set $B \subset \Omega \times (0,t) \times \bR^d$ with $(\bP\times {\rm Leb} \times {\rm Leb})(B)=0$. Therefore, $V_N$ is predictable. Due to the form of $V_N^{(n)}(s,y)$, for any $(\omega,s,y)\in B^c$ such that $|x-y|<t-s$, the sequence $\{u_N^{(n)}(\omega,s,y)\}_{n\geq 0}$ converges to some $u_N'(\omega,s,y)$ as $n\to \infty$, and
$V_{N}(\omega,s,y)=G_{t-s}(x-y)u_N'(\omega,s,y) h(y)^{\frac{p-q}{p}}$. Hence, $u_N'$ is predictable.
An argument based on Fubini's theorem, \eqref{unif-conv} and the uniqueness of the limit shows that $u_N'$ is a modification of $u_N$.

Finally, to prove that $u_N'$ verifies the integral equation \eqref{eq-uN}, we let $n \to \infty$ in \eqref{picard}. We let
$v_{N}^{(n)}(s,y)=G_{t-s}(x-y) \sigma(u_N^{(n)}(s,y))$
and $v_{N}(s,y)=G_{t-s}(x-y) \sigma(u_N'(s,y))$. Then $\{v_N^{(n)}\}_{n\geq 0}$ converges to $v_N$ as $n \to \infty$ in the Daniell mean $\|\cdot\|_{L_N,p}$, which implies the convergence of the stochastic integrals of these processes with respect to $L_N$.

\medskip

We proceed now with the proof of \eqref{aim}. Here $C$ is a constant which depends on $N,p,q,\eta$, and may be different from line to line.
We consider separately the cases $p<1$ and $p\geq 1$.

\medskip

{\em Case 1. $p<1$.}
By Lemma \ref{lem-p-mom} and the Lipschitz property of $\sigma$,
\[
\bE|u_N^{(n+1)}(t,x)-u_{N}^{(n)}(t,x)|^p \leq C\int_0^t \int_{\bR^d} G_{t-s}^{p}(x-y)\bE|u_N^{(n)}(s,y)-u_{N}^{(n-1)}(s,y)|^p h(y)^{p-q} dyds.
\]
Iterating this inequality, we obtain:
\begin{align*}
\bE|u_N^{(n)}(t,x)-u_{N}^{(n-1)}(t,x)|^p & \leq C^n \int_{T_n(t)} \int_{(\bR^d)^n} \prod_{i=1}^{n}G_{t_{i+1}-t_i}^p(x_{i+1}-x_i) \prod_{i=1}^{n}h(x_i)^{p-q}d{\bf x}d{\bf t},
\end{align*}
where ${\bf t}=(t_1,\ldots,t_n)$, ${\bf x}=(x_1,\ldots,x_n)$ and we use the convention $t_{n+1}=t$ and $x_{n+1}=x$.

We use following (generalized) H\"older's inequality: for any non-negative integrable function $f$ on $(\bR^d)^n$, and non-negative functions $g_1,\ldots,g_n$ on $\bR^d$,
\begin{align*}
\int_{(\bR^d)^n} \prod_{i=1}^{n}g_i(x_i)f(x_1,\ldots,x_n) d{\bf x} & \leq \prod_{i=1}^{n} \left(\int_{(\bR^d)^n} g_i^n(x_i) f(x_1,\ldots,x_n)d{\bf x}\right)^{1/n}.
\end{align*}
We apply this to
$f(x_1,\ldots,x_n)=\prod_{i=1}^{n}G_{t_{i+1}-t_i}^p(x_{i+1}-x_i)$ and $g_i=h^{p-q}$ for all $i\leq n$:
\begin{align*}
& \bE|u_N^{(n)}(t,x)-u_{N}^{(n-1)}(t,x)|^p \leq C^n \int_{T_n(t)}\prod_{i=1}^{n} \left(\int_{(\bR^d)^n} \prod_{j=1}^{n}G_{t_{j+1}-t_j}^p(x_{j+1}-x_j) h(x_i)^{n(p-q)}d{\bf x}\right)^{1/n}d{\bf t} \\
& \quad \quad \quad =  C^n \int_{T_n(t)}\prod_{i=1}^{n} \left(\int_{(\bR^d)^n} \prod_{j=1}^{n}G_{t_{j+1}-t_j}^p(x_j) h(x-\sum_{j=i}^n x_j)^{n(p-q)}d{\bf x}\right)^{1/n}d{\bf t} \\
& \quad \quad \quad \leq C^n \frac{1}{n}\sum_{i=1}^{n}\int_{T_n(t)} \int_{(\bR^d)^n} \prod_{j=1}^{n}G_{t_{j+1}-t_j}^p(x_j) h(x-\sum_{j=i}^n x_j)^{n(p-q)}d{\bf x}d{\bf t},
\end{align*}
where for the last line we used the fact that the geometric mean is smaller than the arithmetic mean. We now use the special form of $h$:
\[
h(x)^{n(p-q)}=(1+|x|^{\eta})^{n(p-q)}\leq 2^{n(p-q)}(1+|x|^{n\gamma}) \quad \mbox{for all $x \in \bR^d$},
\]
where $\gamma=\eta(p-q)$. Hence, for any $x \in \bR^d$ with $|x|\leq R$, we have:
\[
h(x-\sum_{j=i}^n x_j)^{n(p-q)} \leq 2^{n(p-q+\gamma)}(1+R^{n\gamma}+|\sum_{j=i}^{n}x_j|^{n\gamma}).
\]
It follows that
\begin{equation}
\label{p-unK}
\bE|u_N^{(n)}(t,x)-u_{N}^{(n-1)}(t,x)|^p \leq C^n \Big\{\big(1+R^{n\gamma}\big) A^{(n)}(t)+B^{(n)}(t)\Big\},
\end{equation}
where
\[
A^{(n)}(t)=\int_{T_n(t)} \int_{(\bR^d)^n} \prod_{j=1}^{n}G_{t_{j+1}-t_j}^p(x_j)d{\bf x}d{\bf t}
\]
and $B^{(n)}(t)=\frac{1}{n}\sum_{i=1}^{n}B_i^{(n)}(t)$, with
\begin{equation}
\label{def-B-in}
B_i^{(n)}(t)=\int_{T_n(t)} \int_{(\bR^d)^n} \prod_{j=1}^{n}G_{t_{j+1}-t_j}^p(x_j) |\sum_{j=i}^n x_j|^{n\gamma}d{\bf x}d{\bf t}.
\end{equation}

Using \eqref{p-G} and Lemma \ref{lem-beta}, we see that:

\[
A^{(n)}(t)=C^n \int_{T_n(t)} \prod_{j=1}^{n}(t_{j+1}-t_j)^{a}d{\bf t}=C^n \frac{t^{(a+1)n}}{\Gamma((a+1)n+1)}
\quad
\mbox{with}
\
a=
\left\{
\begin{array}{ll} 1 & \mbox{if $d=1$} \\
 2-p & \mbox{if $d=2$}
\end{array} \right.
\]

Hence
\begin{equation}
\label{sum-A-finite}
\sum_{n\geq 1}C^n\sup_{t\leq T} A^{(n)}(t) <\infty.
\end{equation}

It remains to show that a similar relation holds for $B^{(n)}(t)$. This involves a study of convolutions of the form $G_t^p*G_s^p$. We treat separately the cases $d=1$ and $d=2$.

In the estimates below, we will use the convention $\prod_{\emptyset}=1$.

\medskip

{\em a) Case $d=1$.}
Using the fact that $G_t^p(x)dx=2^{1-p}G_t(x)$ and $\int_{\bR}G_t(x)=t$, we have
\begin{align*}
B_{i}^{(n)}(t)&=C^n \int_{T_n(t)} \int_{\bR^n} \prod_{j=1}^{n}G_{t_{j+1}-t_j}(x_j)|\sum_{j=i}^n x_j|^{n\gamma}d{\bf x} d{\bf t}\\
&=C^n \int_{T_n(t)} \prod_{j=1}^{i-1}(t_{j+1}-t_{j})\int_{\bR^{n-i+1}} \prod_{j=i}^{n} G_{t_{j+1}-t_{j}}(x_j)|\sum_{j=i}^n x_j|^{n\gamma} dx_i \ldots dx_{n} d{\bf t}.
\end{align*}
If $i=n$, there is no convolution involved, since
\begin{align*}
B_n^{(n)}(t)&=C^n \int_{T_n(t)} \prod_{j=1}^{n-1}(t_{j+1}-t_{j})\int_{\bR}  G_{t-t_{n}}(x_n)|x_n|^{n\gamma} dx_{n} d{\bf t} \\
&=C^n \int_{T_n(t)} \prod_{j=1}^{n-1}(t_{j+1}-t_{j})(t-t_n)^{n\gamma+1}d{\bf t}\leq \frac{C^n t^{n(\gamma+2)}}{n!}.
\end{align*}
If $i\leq n-1$, we use the change of variable $x_i \to z=\sum_{j=i}^{n}x_j$, followed by the semigroup-type property \eqref{semigroup-1}. We obtain:
\begin{align*}
B_i^{(n)}(t) &=C^n \int_{T_n(t)} \prod_{j=1}^{i-1}(t_{j+1}-t_{j})\int_{\bR^{n-i+1}} \prod_{j=i+1}^{n}G_{t_{j+1}-t_{j}}(x_{j}) \\
& \quad \quad \quad \qquad \qquad G_{t_{i+1}-t_i}(z-\sum_{j=i+1}^n x_j)|z|^{n\gamma} dz dx_{i+1} \ldots dx_{n} d{\bf t}\\
&=C^n \int_{\bR}|z|^{n\gamma} \int_{T_n(t)} \prod_{j=1}^{i-1}(t_{j+1}-t_{j}) \int_{\bR^{n-i-1}} \prod_{j=i+2}^{n}  G_{t_{j+1}-t_{j}}(x_{j})
\\
& \quad \quad \quad \quad \quad \quad (G_{t_{i+2}-t_{i+1}}* G_{t_{i+1}-t_i})(z-\sum_{j=i+2}^n x_{j})dx_{i+2}\ldots dx_n d{\bf t}dz \\
& \leq C^n \int_{\bR}|z|^{n\gamma}\int_{T_n(t)} \prod_{j=1}^{i-1}(t_{j+1}-t_{j})(t_{i+2}-t_i)
\int_{\bR^{n-i-2}} \prod_{j=i+3}^{n}  G_{t_{j+1}-t_{j}}(x_{j})
\\
& \quad \quad \quad \quad \quad \quad (G_{t_{i+3}-t_{i+2}}* G_{t_{i+2}-t_i})(z-\sum_{j=i+3}^n x_{j})dx_{i+3}\ldots dx_n d{\bf t}dz.
\end{align*}
We use again inequality \eqref{semigroup-1}. We continue in this manner. At the last step, we obtain:
\begin{align*}
B_i^{(n)}(t) & \leq C^n \int_{\bR}|z|^{n\gamma}\int_{T_n(t)} \prod_{j=1}^{i-1}(t_{j+1}-t_j) (t_{i+2}-t_i) \ldots (t_n-t_i)
(G_{t-t_n}*G_{t_n-t_i})(z) d{\bf t}dz \\
&\leq C^n \int_{T_n(t)} \prod_{j=1}^{i-1}(t_{j+1}-t_j) (t_{i+2}-t_i) \ldots (t_n-t_i) (t-t_i)\left( \int_{\bR}|z|^{n\gamma} G_{t-t_i}(z) dz \right) d{\bf t}\\
& = C^n  \int_{T_n(t)} \prod_{j=1}^{i-1}(t_{j+1}-t_j) (t_{i+2}-t_i) \ldots (t_n-t_i) (t-t_i)^{n\gamma+2} d{\bf t} \\
& \leq C^n t^{n-i+n\gamma+1}\int_{T_n(t)} \prod_{j=1}^{i-1}(t_{j+1}-t_j)d{\bf t}= \frac{C^n t^{n(\gamma+2)}}{\Gamma(i+n)} \leq \frac{C^n t^{n(\gamma+2)}}{n!},
\end{align*}
using Lemma \ref{lem-beta} and the monotonicity of the $\Gamma$ function. Hence
\begin{equation}
\label{ineq-B}
\sum_{n\geq 1}C^n\sup_{t\leq T} B^{(n)}(t) <\infty.
\end{equation}

\medskip

{\em b) Case $d=2$.} If $i=n$, using relations \eqref{p-G} and \eqref{Gp-gamma} and Lemma \ref{lem-beta}, we obtain:
\begin{align*}
B_n^{(n)}(t)&=C^{n-1} \int_{T_n(t)}  \prod_{j=1}^{n-1}(t_{j+1}-t_j)^{2-p}\int_{\bR^2}G_{t-t_n}^p(x_n)|x_n|^{n\gamma}dx_n d{\bf t}\\
&\leq C^n \int_{T_n(t)} \prod_{j=1}^{n-1}(t_{j+1}-t_j)^{2-p}(t-t_n)^{2-p+n\gamma}d{\bf t}\\
&=C^n \frac{\Gamma(3-p)^{n-1}\Gamma(n\gamma +3-p)}{\Gamma(n(3-p+\gamma)+1)}t^{n(3-p+\gamma)} \leq C^n \frac{1}{(n!)^{3-p}}t^{n(3-p+\gamma)}.
\end{align*}
(For the last inequality, we used Stirling's formula.) If $i=n-1$, then by \eqref{p-G}, we have
\begin{align*}
B_{n-1}^{(n)}(t) &= C^{n-2} \int_{T_n(t)} \prod_{j=1}^{n-2}(t_{j+1}-t_j)^{2-p}\int_{(\bR^2)^2}G_{t-t_n}^p(x_n)
G_{t_n-t_{n-1}}^p(x_{n-1})|x_{n-1}+x_{n}|^{n\gamma}dx_{n-1} dx_n d{\bf t}\\
&=C^{n-2}\int_{\bR^2}|z|^{n\gamma}\int_{T_n(t)} \prod_{j=1}^{n-2}(t_{j+1}-t_j)^{2-p} (G_{t-t_n}^p * G_{t_n-t_{n-1}}^p)(z)d{\bf t}dz.
\end{align*}
Fix an arbitrary $r \in (\frac{1}{2},1)$. Since $p<1<2r$, we can
pass from $G^p$ to $G^{2r}$, using \eqref{G-pq-ineq}. So,
\[
B_{n-1}^{(n)}(t) \leq C^{n} \int_{\bR^2}|z|^{n\gamma} \int_{T_{n}(t)} \prod_{j=1}^{n-2}(t_{j+1}-t_j)^{2-p}(t-t_n)^{2r-p} (t_n-t_{n-1})^{2r-p}(G_{t-t_n}^{2r} * G_{t_n-t_{n-1}}^{2r})(z) d{\bf t}dz.
\]
We now use estimate \eqref{I-rt-z} for the $dt_n$ integral on $(t_{n-1},t)$, followed by \eqref{Gp-gamma}. We obtain:
\begin{align*}
B_{n-1}^{(n)}(t) & \leq C^{n+1}\int_{T_{n-1}(t)}
\prod_{j=1}^{n-2}(t_{j+1}-t_j)^{2-p}(t-t_{n-1})^{2(r-p+1)}
\left(\int_{\bR^2} |z|^{n\gamma} G_{t-t_{n-1}}^{2r-1}(z)dz \right) dt_1 \ldots dt_{n-1}\\
&\leq  C^{n+2}\int_{T_{n-1}(t)}
\prod_{j=1}^{n-2}(t_{j+1}-t_j)^{2-p}
 (t-t_{n-1})^{5-2p+n\gamma}dt_1 \ldots dt_{n-1}\\
 &=C^{n+2} \frac{\Gamma(3-p)^{n-2}\Gamma(n\gamma+6-2p)}{\Gamma(n(3-p+\gamma)+1)}
 t^{n(3-p+\gamma)} \leq C^n \frac{1}{(n!)^{3-p}}t^{n(3-p+\gamma)},
\end{align*}
using Lemma \ref{lem-beta} and Stirling's formula.

If $i\leq n-2$ and $n\geq 3$, we use \eqref{p-G} and the change of variable $x_i \mapsto z=\sum_{j=i}^n x_j$:
\begin{align*}
B_{i}^{(n)}(t) &= C^{i-1} \int_{T_n(t)} \prod_{j=1}^{i-1}(t_{j+1}-t_j)^{2-p}\int_{(\bR^2)^{n-i+1}}
\prod_{j=i}^{n} G_{t_{j+1}-t_{j}}^p(x_{j})|\sum_{j=i}^n x_{j}|^{n\gamma}dx_{i} \ldots dx_n d{\bf t}\\
&=C^{i-1}\int_{\bR^2}|z|^{n\gamma}\int_{T_n(t)} \prod_{j=1}^{i-1}(t_{j+1}-t_j)^{2-p} \int_{(\bR^2)^{n-i-1}}\prod_{j=i+2}^{n} G_{t_{j+1}-t_{j}}^p(x_{j}) \\
& \qquad \quad \quad  (G_{t_{i+2}-t_{i+1}}^p * G_{t_{i}-t_{i+1}}^p)(z-\sum_{j=i+2}^n x_j)dx_{i+2} \ldots dx_n d{\bf t}dz.
\end{align*}
In the convolution above, we pass from $G^p$ to $G^{2r}$, using \eqref{G-pq-ineq}. We get:
\begin{align*}
& B_{i}^{(n)}(t)  \leq C^{i+1} \int_{\bR^2} |z|^{n\gamma} \int_{\{0<t_1<\ldots<t_i<t_{i+2}<\ldots<t_n<t\}} \prod_{j=1}^{i-1}(t_{j+1}-t_j)^{2-p} \int_{(\bR^2)^{n-i-1}} \prod_{j=i+2}^{n} G_{t_{j+1}-t_{j}}^p(x_{j})\\
& \quad \quad \quad \left(\int_{t_i}^{t_{i+2}}(t_{i+2}-t_{i+1})^{2r-p} (t_{i+1}-t_{i})^{2r-p} (G_{t_{i+2}-t_{i+1}}^{2r} * G_{t_{i}-t_{i+1}}^{2r})(z-\sum_{j=i+2}^n x_j) dt_{i+1}\right) \\
& \quad \quad \quad \quad \quad \quad dx_{i+2} \ldots dx_n   dt_1 \ldots dt_i dt_{i+2}\ldots dt_n dz.
\end{align*}
We use estimate \eqref{I-rt-z} for the $dt_{i+1}$ integral on $(t_{i},t_{i+2})$. We obtain:
\begin{align*}
B_i^{(n)}(t) & \leq C^{i+2} \int_{\bR^2} |z|^{n\gamma} \int_{\{0<t_1<\ldots<t_i<t_{i+2}<\ldots<t_n<t\}} \prod_{j=1}^{i-1}(t_{j+1}-t_j)^{2-p} (t_{i+2}-t_i)^{2(r-p+1)}\\
& \int_{(\bR^2)^{n-i-1}} \prod_{j=i+2}^{n} G_{t_{j+1}-t_{j}}^p(x_{j})
 G_{t_{i+2}-t_i}^{2r-1}(z-\sum_{j=i+2}^n x_j)dx_{i+2} \ldots dx_n
dt_1 \ldots dt_i dt_{i+2}\ldots dt_n dz.
\end{align*}
We now pass from $G_{t_{i+3}-t_{i+2}}^p(x_{i+2})$ to $G_{t_{i+3}-t_{i+2}}^{2r}(x_{i+2})$, again using \eqref{G-pq-ineq}. Hence
\begin{align*}
& B_i^{(n)}(t)  \leq C^{i+3} \int_{\bR^2} |z|^{n\gamma} \int_{\{0<t_1<\ldots<t_i<t_{i+2}<\ldots<t_n<t\}} \prod_{j=1}^{i-1}(t_{j+1}-t_j)^{2-p} (t_{i+2}-t_i)^{2(r-p+1)}
\\
& \quad (t_{i+3}-t_{i+2})^{2r-p}
 \int_{(\bR^2)^{n-i-2}} \prod_{j=i+3}^{n} G_{t_{j+1}-t_{j}}^p(x_{j})
(G_{t_{i+3}-t_{i+2}}^{2r}* G_{t_{i+2}-t_i}^{2r-1})(z-\sum_{j=i+3}^n x_j) \\
& \qquad \qquad \qquad \qquad dx_{i+3}\ldots dx_n dt_1 \ldots dt_i dt_{i+2}\ldots dt_n dz \\
& \leq C^{i+3} \int_{\bR^2} |z|^{n\gamma} \int_{\{0<t_1<\ldots<t_i<t_{i+3}<\ldots<t_n<t\}} \prod_{j=1}^{i-1}(t_{j+1}-t_j)^{2-p} (t_{i+3}-t_i)^{4r-3p+2}\\
& \quad
\int_{(\bR^2)^{n-i-2}} \prod_{j=i+3}^{n} G_{t_{j+1}-t_{j}}^p(x_{j}) \left(\int_{t_i}^{t_{i+3}} (G_{t_{i+3}-t_{i+2}}^{2r}* G_{t_{i+2}-t_i}^{2r-1})(z-\sum_{j=i+3}^n x_j) dt_{i+2}\right) \\
& \qquad \qquad \qquad \qquad dx_{i+3}\ldots dx_n   dt_1 \ldots dt_i dt_{i+3}\ldots dt_n dz.
\end{align*}
By Lemma \ref{lem-BNZ2}, the $dt_{i+2}$ integral is bounded by
$C (t_{i+3}-t_i)^{4(1-r)}1_{\{|z-\sum_{j=i+3}^n x_j|<t_{i+3}-t_i\}}$.
Note that each term $G_{t_{j+1}-t_j}^p(x_j)$ contains the indicator $1_{\{|x_j|<t_{j+1}-t_j\}}$, and
\[
\prod_{j=i+3}^n 1_{\{|x_j|<t_{j+1}-t_j\}} 1_{\{|z-\sum_{j=i+3}^n x_j|<t_{i+3}-t_i\}} \leq 1_{\{|z|<t-t_i\}}.
\]
Using the fact that $\int_{\{|z|<t-t_i\}} |z|^{n\gamma}dz \leq C (t-t_i)^{n\gamma+2} \leq C t^{n\gamma+2}$, followed by \eqref{p-G} and Lemma \ref{lem-beta}, we obtain:
\begin{align*}
B_i^{(n)}(t)  &\leq C^{i+5} t^{n\gamma+2} \int_{\{0<t_1<\ldots<t_i<t_{i+3}<\ldots<t_n<t\}} \prod_{j=1}^{i-1}(t_{j+1}-t_j)^{2-p} (t_{i+3}-t_i)^{6-3p}\\
& \quad \quad \quad \int_{(\bR^2)^{n-i-2}} \prod_{j=i+3}^n G_{t_{j+1}-t_j}^p(x_j) dx_{i+3}\ldots dx_n dt_1 \ldots dt_i dt_{i+3} \ldots dt_n\\
&= C^{n+3} t^{n\gamma+2} \int_{\{0<t_1<\ldots<t_i<t_{i+3}<\ldots<t_n<t\}} \prod_{j=1}^{i-1}(t_{j+1}-t_j)^{2-p} (t_{i+3}-t_i)^{6-3p}\\
& \quad \quad \quad \qquad \qquad \prod_{j=i+3}^n
(t_{j+1}-t_j)^{2-p}  dt_1 \ldots dt_i dt_{i+3} \ldots dt_n\\
&=C^{n+3} t^{n\gamma+2} \frac{\Gamma(3-p)^{n-3}\Gamma(7-3p)}{\Gamma(n(3-p)-2+1)}t^{n(3-p)-2} \leq C^n \frac{1}{(n!)^{3-p}}t^{n(3-p+\gamma)}.
\end{align*}
For the last inequality, we used the fact that for any $a>0$ and $b\in \bR$, $\Gamma(an+b+1) \geq C_{a,b}^n (n!)^a$ , where $C_{a,b}>0$ is a constant depending on $a$ and $b$. This is a consequence of Stilrling's formula and the fact that $\Gamma(x+b) \sim \Gamma(x)x^b$ as $x\to \infty$.

To summarize, we have proved that:
\[
B_i^{(n)} \leq C^n \frac{1}{(n!)^{3-p}}t^{n(3-p+\gamma)} \quad \mbox{for all $i=1,\ldots,n$}.
\]
This shows that relation \eqref{ineq-B} also holds in the case $d=2$.

\medskip

{\em Case 2. Assume that $p\geq 1$.} In this case,
we work with the function $G_{t,p}(x)=G_t^p(x)+G_t(x)$ instead of $G_t^p(x)$.
Relation \eqref{p-unK} holds using the same argument as in case $p<1$, but with constants $A^{(n)}(t)$ and $B_i^{(n)}(t)$ given by:
\begin{align*}
A^{(n)}(t)&=\int_{T_n(t)} \int_{(\bR^d)^n} \prod_{j=1}^{n}G_{t_{j+1}-t_j,p}(x_j)d{\bf x}d{\bf t}\\
B_i^{(n)}(t)&=\int_{T_n(t)} \int_{(\bR^d)^n} \prod_{j=1}^{n}G_{t_{j+1}-t_j,p}(x_j) |\sum_{j=i}^n x_j|^{n\gamma}d{\bf x}d{\bf t}.
\end{align*}

We treat separately the cases $d=1$ and $d=2$.

\medskip

a) {\em Case $d=1$.} Since $G_{t}^p(x)=2^{1-p}G_t(x)$, we can use the same argument as for $p<1$.

\medskip

b) {\em Case $d=2$.} This case is treated similarly to the case $p<1$. We use the following properties. For any $t>0$, $\int_{\bR^2}G_{t,p}(x)dx=c_p t^{2-p}+t$, with $c_p=\frac{(2\pi)^{1-p}}{2-p}$. For any $\gamma>0$,
\[
\int_{\bR^2}G_{t,p}(x)|x|^{\gamma}dx \leq c_{p,T} \, t^{2-p+\gamma} \quad \mbox{for all} \ t \leq T,
\]
with $c_{p,T}=c_p+T^{p-1}$. For any $r \in (\frac{1}{2},1)$ with $p<2r$, we have
\[
G_{t,p}(x)\leq c_{r,p,T} t^{2r-p} G_t^{2r}(x) \quad \mbox{for any} \ t \leq T,
\]
with $c_{r,p,T}=(2\pi)^{2r-1}(1+T^{p-1})$.
\end{proof}

The following result gives the existence of solution to equation \eqref{wave}.

\begin{theorem}
\label{exist-th}
Under the hypotheses of Theorem \ref{exist-th-K},
equation \eqref{wave-eq} has a solution $u$. Moreover, there exists a sequence $(\tau_N)_{N \geq 1}$ of stopping times with $\tau_{N} \uparrow \infty$ a.s. such that
\begin{equation}
\label{stop}
\sup_{t \in [0,T]} \sup_{|x|\leq R}\bE\big[|u(t,x)|^p 1_{\{t \leq \tau_N\}} \big]<\infty,
\end{equation}
for any $T>0$, $R>0$ and $N \geq 1$, where $p=2$ if $d=1$ and $p$ is the constant from Assumption A if $d=2$.
\end{theorem}

\begin{proof} Let $h(x)=1+|x|^{\eta}$, where $\eta>d/q$ and $q$ is the constant from Assumption A. Let
\begin{equation}
\label{def-tau}
\tau_N=\inf \left\{t>0; J([0,t] \times \{(x,z);|z|>N h(x)\})>0\right\},
\end{equation}
Since the set $[0,t] \times \{(x,z);|z|>N h(x)\}$ is unbounded, $J([0,t] \times \{(x,z);|z|>N h(x)\})$ may be infinite, which would mean that $\tau_N$ is not well-defined. This delicate issue is addressed in the proof of Lemma 3.2 of \cite{chong17-SPA}, where it is proved that for any $T>0$ and $N\geq 1$, with probability 1,
\begin{equation}
\label{eq1}
\mbox{$J$ has finitely many points in $[0,T] \times \{(x,z);|z|>N h(x)\}$.}
\end{equation}
Since the details of this argument are missing in \cite{chong17-SPA}, we include them below. Let $V_0=\emptyset$ and $V_n=\{x \in \bR^d; |x|\leq b_n\}$ be such that ${\rm Leb}(V_n)=n$. Let $U_n=V_n\verb2\2 V_{n-1}$ for any $n\geq 1$. Then $(U_n)_{n\geq 1}$ is a partition of $\bR^d$ with ${\rm Leb}(U_n)=1$. For fixed $T>0$ and $N \geq 1$, let $$F_n=[0,T] \times \{(x,z)\in U_n \times \bR; |z|>Nh(x)\}.$$ Under the condition $\eta>d/q$, it can be proved that
$\sum_{n\geq 1}P(J(F_n)>0)<\infty$. By the Borel-Cantelli lemma, it follows that with probability 1, there exists $n_0 \geq 1$ such that $J(F_n)=0$ for all $n\geq n_0$. This means that, with probability 1, $J$ has no points in the set
$$\bigcup_{n\geq n_0}F_n=[0,T]\times \{(x,z); |x|>b_{n_0-1},|z|>Nh(x)\}.$$
Since $J$ has finitely many points in the set $[0,T]\times \{(x,z) ; |x|\leq b_{n_0-1},|z|>Nh(x)\}$, relation \eqref{eq1} follows.

The argument above shows that with probability 1, for any $T\in \bQ_{+}$, there exists $N_0$ large enough such that for all $N\geq N_0$,
$J([0,T]\times \{(x,z); |z|>Nh(x)\})=0$ (and therefore $\tau_N>T$). This proves that $\tau_N \uparrow \infty$ a.s. when $N \to \infty$.

Let $u_N$ be the process given by Theorem \ref{exist-th-K}.
As in the last part of the proof of Theorem 3.1 of \cite{chong17-SPA}, it can be proved that for any $N\geq 1$,
$u_N(t,x)=u_{N+1}(t,x)$ a.s. on $\{t \leq \tau_{N}\}$. This argument uses the fact that on the event $\{\tau_N>T\}$
\[
L([0,t]\times A)=L_N([0,t]\times A)=L_{N'}([0,t]\times A)
\]
for any $N'>N,t\in [0,T]$ and $A \in \cB_b(\bR^d)$.
Moreover, the process $u=\{u(t,x);t\geq 0,x \in \bR^d\}$ given by
$ u(t,x)= u_{N}(t,x)$ if $t \leq \tau_{N}$ is a solution to equation \eqref{wave-eq}.
\end{proof}

It was suggested by the referee that the second condition in \eqref{pq-cond} (about the existence of $q$) and special form $h(x)=1+|x|^{\eta}$ may not be needed, and one can consider simply $h(x)=1$. This would be right if the goal would be to prove the existence of solution for the equation with truncated noise. Indeed, this truncation procedure was used in \cite{B14} for the $\alpha$-stable L\'evy noise. The problem is to paste together these solutions to produce a solution to equation \eqref{wave-eq}. We explain this below.
Consider the truncated noise
\[
\bar{L}_{N}(A)=b|A|+\int_{A \times \{|z|\leq 1\}}z \widetilde{J}(dt,dx,dz)+
\int_{A \times \{1<|z|\leq N\}}z J(dt,dx,dz).
\]
Lemma \ref{lem-p-mom} remains valid with $h(x)=1$, if $p \in (0,2]$ is a value such that $\int_{|z|\leq 1}|z|^p\nu(dz)<\infty$.
We have the following result.

\begin{theorem}
a) If $d=1$, equation \eqref{wave} with noise $\bar{L}_{N}$ instead of $L_N$ has a unique solution $\bar{u}_N$, and this solution satisfies:
\[
\sup_{(t,x) \in [0,T]}\bE|\bar{u}_N(t,x)|^2<\infty.
\]
b) If $d=2$ and there exists $p \in (0,2)$ such that $\int_{|z|\leq 1}|z|^p\nu(dz)<\infty$, then equation \eqref{wave} with noise $\bar{L}_{N}$ instead of $L_N$ has a unique solution $\bar{u}_N$, and this solution satisfies:
\[
\sup_{(t,x) \in [0,T]}\bE|\bar{u}_N(t,x)|^p<\infty.
\]
\end{theorem}

\begin{proof} If $d=1$, we let $p=2$. In both cases, $
\int_0^t \int_{\bR^d}G_{t-s}^p(x-y)dyds<\infty$.
Define the Picard iterations: $\bar{u}_{N}^{(0)}(t,x)=w(t,x)$ and
\[
\bar{u}_{N}^{(n+1)}(t,x)=w(t,x)+\int_0^t \int_{\bR^d}G_{t-s}(x-y)\sigma(\bar{u}_{N}^{(n)}(s,y))\bar{L}_{N}(ds,dy) \quad n\geq 0.
\]
Let $H_n(t)=\sup_{t \in [0,T]}\bE|\bar{u}_{N}^{(n)}(t,x)-\bar{u}_{N}^{(n-1)}(t,x)|^p$. By Lemma \ref{lem-p-mom} with $h(x)=1$, we obtain:
\[
H_{n+1}(t)\leq C \int_0^t H_n(s) J(t-s)ds
\]
where $J(t-s)=\int_0^t \int_{\bR^d}\Big(G_{t-s}^p(x-y)+G_{t-s}(x-y)1_{\{p\geq 1\}}\Big)dyds$.
By Lemma 15 of \cite{dalang99}, $\sum_{n\geq 1}\sup_{t\leq T}H_n(t)^{1/p}<\infty$. (In fact, this lemma shows that $H_n(t) \leq C a_n$, where $(a_n)_n$ satisfies $\sum_{n\geq 1}a_n^{1/b}<\infty$ for any $b>0$.) It follows $\{\bar{u}_{N}^{(n)}(t,x)\}_{n\geq 1}$ is a Cauchy sequence in $L^p(\Omega)$, uniformly in $[0,T] \times \bR^d$. We denote by $\bar{u}_N(t,x)$ its limit.

To show that
$\bar{u}_N$ has a predictable modification, let
$\bar{V}_N^{(n)}(s,y)=G_{t-s}(x-y)\bar{u}_N(s,y)$. If $p\in (0,1)$, then
\[
\|\bar{V}_N^{(n)}-\bar{V}_N^{(n-1)}\|_{L^p(\Omega \times (0,t) \times \bR^d)}=\int_0^t \int_{\bR^d}G_{t-s}^p(x-y) \bE|\bar{u}_N^{(n)}(s,y)-\bar{u}_N^{(n-1)}(s,y)|^pdyds\leq C a_n,
\]
and therefore
\[
\sum_{n\geq 1}\|\bar{V}_N^{(n)}-\bar{V}_N^{(n-1)}\|_{L^p(\Omega \times (0,t) \times \bR^d)} \leq C \sum_{n\geq 1}a_n<\infty.
\]
(This relation also holds if $p\in [1,2]$, using the fact that $\sum_{n\geq 1}a_n^{1/p}<\infty$.) The existence of the predictable modification follows as in the proof of Theorem \ref{exist-th-K}.
This modification is the unique solution of equation \eqref{wave} with $L_N$ replaced by $\bar{L}_N$.
\end{proof}

\begin{remark}
{\rm The natural stopping time associated with the truncation $h(x)=1$ is
\[
\bar{\tau}_N=\inf\{t>0; J([0,t] \times \bR^d \times \{|z|>N\})>0 \}
\]
Unfortunately, unlike \eqref{eq1}, we could not find an argument to show that $J$ has finitely many points in the (unbounded) set $[0,t] \times \bR^d \times \{|z|>N\}$. Hence, $\bar{\tau}_N$ may not be well-defined, as mentioned in Remark 27 of \cite{B14}. Because of this, we could not proceed as in the last part of the proof of Theorem \ref{exist-th} to obtain the existence of a solution of equation \eqref{wave-eq}, based on the solutions $\bar{u}_N$. This explains why we cannot drop the condition about $q$ in \eqref{pq-cond}: the existence of $q$ allows us to choose a suitable $\eta$ such that (30) holds, and hence $\tau_N$ is well-defined.
}
\end{remark}

\section{Path properties of the solution}

In this section, we fix $T>0$ and we study the path properties of the solution $u$ to equation \eqref{wave-eq} on the interval $[0,T]$. More precisely, we will show that the process $\{u(t,\cdot)\}_{t \in [0,T]}$ has a modification which is c\`adl\`ag (i.e. right-continuous with left limits) in a suitable space.

We say that two processes $\{X(t)\}_{t \in [0,T]}$ and $\{Y(t)\}_{t\in [0,T]}$ (defined on the same probability space) are {\em modifications} of each other if
$\bP(X(t)=Y(t))=1$ for almost all $t \in [0,T]$.

\medskip

Let $h(x)=1+|x|^{\eta}$ for some $\eta>d/q$, and
\begin{equation}
\label{def-tau2}
\tau_N=\inf \left\{t \in [0,T]; J([0,t] \times \{(x,z);|z|>N h(x)\})>0 \right\}.
\end{equation}
Compared with \eqref{def-tau}, the infimum is now taken over $[0,T]$. It follows that with probability $1$, for $N$ large enough, $\tau_N=\infty$ and $u(t,x)=u_N(t,x)$ for all $t \in [0,T]$. So, it suffices to study the path properties of $u_N$ for fixed $N \geq 1$.

\medskip

We consider the fractional Sobolev space of order $r \in \bR$:
\[
H^{r}(\bR^2)=\{f \in \cS'(\bR^2); \|f\|_{H^r(\bR^2)}^2:=\int_{\bR^2}|\cF f(\xi)|^2(1+|\xi|^2)^{r} d\xi <\infty\}
\]
and the local fractional Sobolev space of order $r$:
\[
H_{\rm loc}^r(\bR^2)=\{f \in \cS'(\bR^2); \varphi f \in H^r(\bR^2) \ \mbox{for all} \ \varphi \in C_c^{\infty}(\bR^2)\}
\]
We say that $f_n \to f$ in $H_{\rm loc}^r(\bR^2)$ if $f_n \varphi \to f \varphi$ in $H^r(\bR^2)$, for any $\varphi \in C_c^{\infty}(\bR^2)$. The embedding map of $H^r(\bR^2)$ into $H_{\rm loc}^r(\bR^2)$ is continuous.

\medskip

Before investigating the path properties of the solution
$u$, we need to examine the regularity of the function $G$.
A key estimate is the following: for any $t>0$ and $\xi \in \bR$,
\begin{equation}
\label{key}
\frac{\sin^2(t|\xi|)}{|\xi|^2} \leq 2(t^2 \vee 1)\frac{1}{1+|\xi|^2}.
\end{equation}
Using this inequality, it follows that for any $t>0$,
\begin{equation}
\label{FG-bound}
\int_{\bR}|\cF G_t(\xi)|^2 (1+|\xi|^2)^r d\xi \leq 2(t^2 \vee 1) \int_{\bR} \left(\frac{1}{1+|\xi|^2}\right)^{1-r}d\xi.
\end{equation}
Therefore, $G_t \in H^r(\bR^d)$ for any $r<1-d/2$ and $t>0$.

\medskip

By convention, we let $G_t(x)=0$ for any $t<0$ and $x \in \bR^d$. To define the function $G$ at time $0$, we consider the limit as $t \to 0+$. We treat separately the cases $d=1$ and $d=2$. We start with the case $d=2$ since it is more involved.

\subsection{Case $d=2$}

In this case, $G_t \in H^r(\bR^2)$ for any $r<0$ and $t>0$.
On the other hand, $G_0=\delta_0$ (the Dirac delta distribution at $0$) since
\[
G_0(x):=\lim_{t\to 0+}G_t(x)=
\left\{
\begin{array}{ll}
0 & \mbox{if $x \not =0$} \\
\infty & \mbox{if $x=0$}
\end{array} \right.
\]
The Dirac delta distribution $\delta_x$ is in
$H^r(\bR^2)$ if and only if $r<-1$, since $\cF \delta_x(\xi)=e^{-i \xi \cdot x}$.

\medskip

The following result will allow us to analyze the compound-Poisson component of $u^N$.

\begin{lemma}
\label{lem-G-t0}
If $d=2$, for any $t_0 \in [0,T]$, $x_0 \in \bR^2$, the map $[0,T] \ni t\mapsto G_{t-t_0}(\, \cdot-x_0)$ is c\`adl\`ag in $H^r(\bR^2)$ for any $r<-1$.
\end{lemma}

\begin{proof}
We have to prove that the function $F:[0,T] \to H^r(\bR^2)$ given by
\[
F(t)=
\left\{
\begin{array}{ll}
G_{t-t_0}(\, \cdot -x_0) & \mbox{if $t >t_0$} \\
\delta_{x_0} & \mbox{if $t=t_0$}\\
0 & \mbox{if $t<t_0$}
\end{array} \right.
\]
is c\`adl\`ag. Note that $F$ is continuous at any point $t>t_0$, since
\begin{align*}
\|F(t+h)-F(t)\|_{H^r(\bR^2)}^2 &= \int_{\bR^2}\left|\frac{\sin((t+h)|\xi|)}{|\xi|}-
\frac{\sin(t|\xi|)}{|\xi|} \right|^2 (1+|\xi|^2)^{r}d\xi\\
&=4 \int_{\bR^2} \frac{(1+|\xi|^2)^{r}}{|\xi|^2}\sin^2\left(\frac{h|\xi|}{2}\right) \cos^2 \frac{(2t+h)|\xi|}{2}d\xi \to 0
\end{align*}
as $h \to 0$, by the dominated convergence theorem. To justify the application of this theorem we use the fact that $\sin^2(\frac{h|\xi|}{2}) \leq C \frac{|\xi|^2}{1+|\xi|^2}$ and $\int_{\bR^2}(1+|\xi|^2)^{r-1}d\xi<\infty$.

 Clearly, $F$ is continuous at any point $t<t_0$. $F$ is right-continuous at $t_0$ since
\begin{align*}
\|F(t_0+h)-F(t_0)\|_{H_r^2(\bR^2)}&=\int_{\bR^2} |\cF G_h(\cdot-x_0)(\xi)-\cF \delta_{x_0}(\xi)|^2 (1+|\xi|^2)^r d\xi\\
&=\int_{\bR^2}\left|e^{-i \xi \cdot x_0} \frac{\sin(h|\xi|)}{|\xi|}-e^{-i\xi \cdot x_0} \right|^2 (1+|\xi|^2)^r d\xi \to 0
\end{align*}
as $h \to 0+$, by the dominated convergence theorem. Clearly, $F$ has left limit $0$ at $t_0$.
\end{proof}

We consider first the case of bounded function $\sigma$. The compact support property of $G$ turns out to be very useful, and compensates for its lack of smoothness.

\begin{theorem}
\label{th-sigma-bded}
Suppose that $d=2$ and Assumption A holds with $p<2$. Assume that $\sigma$ is bounded. Let $\{u(t,x);t\in [0,T],x \in \bR^2\}$ be the solution to equation \eqref{wave-eq} on the interval $[0,T]$, constructed in Theorem \ref{exist-th} but with stopping times $(\tau_N)_{N\geq 1}$ defined by \eqref{def-tau2}. Then the process $\{u(t,\cdot)\}_{t \in [0,T]}$ has a c\`adl\`ag modification with values in $H_{\rm loc}^r(\bR^2)$, for any $r<-1$.
\end{theorem}

\begin{proof}
We follow the lines of the proof of Proposition 2.14 of \cite{CDH19} for the heat equation. As mentioned above, it suffices to prove the result for $u_N$, for arbitrary $N\geq 1$. We consider separately the cases $p\geq 1$ and $p<1$, $p$ being the constant from Assumption A.

\medskip

{\em Case 1. $p\geq 1$.} Recall that $u_N$ satisfies the integral equation \eqref{eq-uN}. Using decomposition \eqref{decomp-LN}
of $L_N$, we write:
\begin{align}
\nonumber
u_N(t,x)&=w(t,x)+\int_0^t \int_{\bR^2}G_{t-s}(x-y)\sigma(u_N(s,y)) L^M(ds,dy)+ \\
\nonumber
& \quad \int_0^t \int_{\bR^2}G_{t-s}(x-y)\sigma(u_N(s,y)) L_N^P(ds,dy)+b\int_0^t \int_{\bR^2}G_{t-s}(x-y)\sigma(u_N(s,y)) dyds \\
\label{decomp-uN}
& =: w(t,x)+u_N^1(t,x)+u_N^2(t,x)+u_N^3(t,x).
\end{align}

\medskip

\underline{We treat $w(t,x)$.}
Since $w$ is jointly continuous and bounded, the map $t \mapsto w(t,\cdot)$ is continuous in $H_{\rm loc}^r(\bR^2)$ (see the proof of Lemma 2.13 of \cite{CDH19}).

\medskip

\underline{We treat $u_N^1(t,x)$.} Let $A\geq T$ be arbitrary and $K=\{y\in \bR^2;|y| \leq 2A\}$. We use the decomposition $u_N^1(t,x)=u_N^{1,1}(t,x)+u_N^{1,2}(t,x)$, where
\begin{align}
\label{def-uN-11}
u_N^{1,1}(t,x)&=\int_0^t \int_{K} G_{t-s}(x-y)\sigma(u_N(s,y))L^M(ds,dy)\\
\label{def-uN-12}
u_N^{1,2}(t,x)&=\int_0^t \int_{K^c} G_{t-s}(x-y)\sigma(u_N(s,y))L^M(ds,dy).
\end{align}
Suppose that $|x|\leq A$.
Note that $G_{t-s}(x-y)$ contains the indicator of the set $\{y;|x-y|<t-s\}$ and any element in this set satisfies:
\[
|y|\leq |y-x|+|x|\leq t-s+A\leq T+A\leq 2A.
\]
This shows that $u_{N}^{1,2}(t,\cdot)1_{\{|x|\leq A\}}=0$ for any $A\geq T$.

It remains to examine $u_N^{1,1}(t,x)$. Similarly to (2.33) of \cite{CDH19}, for any $\varphi \in \cS(\bR^2)$, we have
\[
\langle \cF u^{1,1}(t,\cdot),\varphi \rangle=\int_{\bR^2}\left( \int_0^t \int_{K} e^{-i \xi \cdot y} \frac{\sin((t-s)|\xi|)}{|\xi|}
\sigma(u_N(s,y))L^M(ds,dy) \right)\varphi(\xi)d\xi.
\]
The proof of this relation is based on applying twice the stochastic Fubini theorem given by Theorem A.3.1 of \cite{CDH19}. For the first  application of this theorem, we need to check that
\[
\int_{\bR^2} |\cF \varphi(x)| \left( \int_0^t \int_{K} \int_{\{|z|\leq 1\}} G_{t-s}^p(x-y) \bE|\sigma(u_N(s,y))|^p |z|^p\nu(dz) dyds\right)^{1/p}dx<\infty.
\]
This follows using Assumption A, \eqref{p-mom-finite}, and the fact that $\int_{\bR^2}G_{t-s}^p(x-y)dy=c_p (t-s)^{2-p}$. For the second application, we need to check that
\[
\int_{\bR^2} |\varphi(\xi)| \left( \int_0^t \int_{K} \int_{\{|z|\leq 1\}} \left|e^{-i \xi \cdot y} \frac{\sin((t-s)|\xi|)}{|\xi|}\right|^p  \bE|\sigma(u_N(s,y))|^p |z|^p\nu(dz) dyds\right)^{1/p}d\xi<\infty,
\]
which follows using Assumption A, \eqref{p-mom-finite}, and the fact that $\int_0^t \left|\frac{\sin((t-s)|\xi|)}{|\xi|}\right|^pds\leq \int_0^t (t-s)^p ds$.

It follows that the Fourier transform of $u_{N}^{1,1}(t,\cdot)$ is the following function:
\begin{equation}
\label{Fourier-uN-11}
a_{\xi}(t):=\cF u_{N}^{1,1}(t,\cdot)(\xi)=\int_0^t \int_{K} e^{-i \xi \cdot y} \frac{\sin((t-s)|\xi|)}{|\xi|}
\sigma(u_N(s,y))L^M(ds,dy).
\end{equation}
Note that for any $t \in [0,T]$, $u_N^{1,1}(t,\cdot) \in H^r(\bR^2)$ with probability $1$, since
\begin{align*}
&\bE\left[\int_{\bR^2}|a_{\xi}(t)|^2(1+|\xi|^2)^{r}d\xi\right] \\
&=\int_{\bR^2} \bE \left|\int_0^t \int_{K} \int_{|z|\leq 1}e^{-i \xi \cdot y} \frac{\sin((t-s)|\xi|)}{|\xi|}
\sigma(u_N(s,y))z\widetilde{J}(ds,dy,dz) \right|^2 (1+|\xi|^2)^{r}d\xi\\
&= \int_{\bR^2} \left(\int_0^t \int_{K} \int_{|z|\leq 1}\frac{\sin^2((t-s)|\xi|)}{|\xi|^2}
\bE|\sigma(u_N(s,y))|^2 |z|^2 \nu(dz) dyds\right)
(1+|\xi|^2)^{r}d\xi\\
& \leq C \int_{\bR^2} \left(\int_0^t \frac{\sin^2((t-s)|\xi|)}{|\xi|^2} ds\right)(1+|\xi|^2)^{r}d\xi<\infty,
\end{align*}
where for the last inequality above, we used the fact that $\sigma$ is bounded (since relation \eqref{p-mom-finite} may not hold for $p=2$). The last integral is finite since $r<0$ (using \eqref{key}).

We prove that $\{u_N^{1,1}(t,\cdot)\}_{t\in [0,T]}$ is stochastically continuous as $H^r(\bR^2)$-valued process, i.e.
\begin{equation}
\label{stoc-cont}
\bE\Big[\|u_N^{1,1}(t+h,\cdot)-u_N^{1,1}(t,\cdot)\|_{H^r(\bR^2)}^2\Big]=0 \quad \mbox{as} \quad h\to 0.
\end{equation}
To see this, note that
\begin{align}
\nonumber
& |a_{\xi}(t+h)-a_{\xi}(t)|^2 \leq \\
\nonumber
& 2 \left\{ \left|
\int_0^t \int_{K} e^{-i \xi \cdot y} \frac{\sin((t+h-s)|\xi|)-\sin((t-s)|\xi|)}{|\xi|} \sigma(u_N(s,y)) L^M(ds,dy) \right|^2 + \right. \\
\label{a-incr-1}
& \quad \left. \left|
\int_{t}^{t+h} \int_{K} e^{-i \xi \cdot y} \frac{\sin((t+h-s)|\xi|)}{|\xi|} \sigma(u_N(s,y)) L^M(ds,dy) \right|^2 \right\}.
\end{align}
Hence, using again the fact that $\sigma$ is bounded, we have
\begin{align*}
&\bE\Big[|a_{\xi}(t+h)-a_{\xi}(t)|^2\Big] \leq \\
& 2 \left\{
\int_0^t \int_{K} \frac{\sin^2((t+h-s)|\xi|)-\sin((t-s)|\xi|)}{|\xi|^2} \bE|\sigma(u_N(s,y))|^2 |z|^2 \nu(dz)dy ds + \right. \\
& \quad \left.
\int_{t}^{t+h} \int_{K} \frac{\sin^2((t+h-s)|\xi|)}{|\xi|^2} \bE|\sigma(u_N(s,y))|^2 |z|^2 \nu(dz) dyds  \right\}\\
&\leq C \left\{ \int_0^t \frac{\sin^2((t+h-s)|\xi|)-\sin((t-s)|\xi|)}{|\xi|^2}ds+\int_{t}^{t+h} \frac{\sin^2((t+h-s)|\xi|)}{|\xi|^2} ds\right\}\\
&=C \left\{\frac{4}{|\xi|^2} \sin^2 \left(\frac{h|\xi|}{2}\right) \int_0^t \cos^2 \frac{2(t-s)+h}{2}ds+\int_0^h \frac{\sin^2(s|\xi|)}{|\xi|^2}ds\right\}=C|h|^2,
\end{align*}
and therefore, by the dominated convergence theorem,
\[
\bE\Big[\|u_N^{1,1}(t+h,\cdot)-u_N^{1,1}(t,\cdot)\|_{H^r(\bR^2)}^2\Big]=
\int_{\bR^2}\bE\Big[|a_{\xi}(t+h)-a_{\xi}(t)|^2\Big] (1+|\xi|^2)^r d\xi
\to 0 \quad \mbox{as} \quad h \to 0.
\]
To justify the application of this theorem, we use inequality \eqref{key} and the fact that $r<0$.

To conclude that $\{u_N^{1,1}(t,\cdot)\}_{t \in [0,T]}$ has a c\`adl\`ag modification with values in $H^r(\bR^2)$, we apply Theorems 1 and 5 of \cite{GS74}. For this, we need to show that there exists $\delta>0$ such that for any $t \in [0,T]$ and for any $h>0$ such that $t+h,t-h \in [0,T]$,
\begin{equation}
\label{mom-incr}
\bE \Big[ \|u_N^{1,1}(t+h,\cdot)-u_N^{1,1}(t,\cdot)\|_{H^r(\bR^2)}^2
\|u_N^{1,1}(t-h,\cdot)-u_N^{1,1}(t,\cdot)\|_{H^r(\bR^2)}^2\Big] \leq C h^{1+\delta}.
\end{equation}
To prove this, note that
\begin{align}
\nonumber
\bE \Big[ & \|u_N^{1,1}(t+h,\cdot)-u_N^{1,1}(t,\cdot)\|_{H^r(\bR^2)}^2
\|u_N^{1,1}(t-h,\cdot)-u_N^{1,1}(t,\cdot)\|_{H^r(\bR^2)}^2\Big] =\\
\label{bound-E}
& \quad \int_{(\bR^2)^2}\bE\Big[|a_{\xi}(t+h)-a_{\xi}(t)|^2 |a_{\eta}(t-h)-a_{\eta}(t)|^2\Big](1+|\xi|^2)^r (1+|\eta|^2)^r d\xi d\eta.
\end{align}
We combine \eqref{a-incr-1} with the similar bound for the increment over $[t-h,t]$, namely:
\begin{align*}
& |a_{\eta}(t-h)-a_{\eta}(t)|^2 \leq \\
& 2 \left\{ \left|
\int_0^t \int_{K} e^{-i \eta \cdot y} \frac{\sin((t-h-s)|\eta|)-\sin((t-s)|\eta|)}{|\eta|} \sigma(u_N(s,y)) L^M(ds,dy) \right|^2 + \right. \\
& \quad \left. \left|
\int_{t-h}^{t} \int_{K} e^{-i \eta \cdot y} \frac{\sin((t-h-s)|\eta|)}{|\eta|} \sigma(u_N(s,y)) L^M(ds,dy) \right|^2 \right\}.
\end{align*}
It follows that
\begin{equation}
\label{bound-E2}
E_{t,h}(\xi,\eta):=\bE\Big[|a_{\xi}(t+h)-a_{\xi}(t)|^2 |a_{\eta}(t-h)-a_{\eta}(t)|^2\Big] \leq \sum_{i=1}^4 E_{t,h}^i(\xi,\eta),
\end{equation}
where
\begin{align*}
E_{t,h}^1(\xi,\eta)&=\bE\left[\left|
\int_0^t \int_{K} e^{-i \xi \cdot y} \frac{\sin((t+h-s)|\xi|)-\sin((t-s)|\xi|)}{|\xi|} \sigma(u_N(s,y)) L^M(ds,dy) \right|^2 \right. \\
& \quad \quad \left. \left|
\int_0^t \int_{K} e^{-i \eta \cdot y} \frac{\sin((t-h-s)|\eta|)-\sin((t-s)|\eta|)}{|\eta|} \sigma(u_N(s,y)) L^M(ds,dy) \right|^2 \right]\\
E_{t,h}^2(\xi,\eta)&=\bE\left[\left|
\int_0^t \int_{K} e^{-i \xi \cdot y} \frac{\sin((t+h-s)|\xi|)-\sin((t-s)|\xi|)}{|\xi|} \sigma(u_N(s,y)) L^M(ds,dy) \right|^2 \right. \\
& \quad \quad \left. \left|
\int_{t-h}^{t} \int_{K} e^{-i \eta \cdot y} \frac{\sin((t-h-s)|\eta|)}{|\eta|} \sigma(u_N(s,y)) L^M(ds,dy) \right|^2 \right]\\
E_{t,h}^3(\xi,\eta)&=\bE \left[ \left|
\int_{t}^{t+h} \int_{K} e^{-i \xi \cdot y} \frac{\sin((t+h-s)|\xi|)}{|\xi|} \sigma(u_N(s,y)) L^M(ds,dy) \right|^2 \right. \\
& \quad \quad \left. \left|
\int_0^t \int_{K} e^{-i \eta \cdot y} \frac{\sin((t-h-s)|\eta|)-\sin((t-s)|\eta|)}{|\eta|} \sigma(u_N(s,y)) L^M(ds,dy) \right|^2 \right]
\end{align*}
\begin{align*}
E_{t,h}^4(\xi,\eta)&=\bE \left[ \left|
\int_{t}^{t+h} \int_{K} e^{-i \xi \cdot y} \frac{\sin((t+h-s)|\xi|)}{|\xi|} \sigma(u_N(s,y)) L^M(ds,dy) \right|^2 \right. \\
& \quad \quad \left. \left|
\int_{t-h}^{t} \int_{K} e^{-i \eta \cdot y} \frac{\sin((t-h-s)|\eta|)}{|\eta|} \sigma(u_N(s,y)) L^M(ds,dy) \right|^2 \right].
\end{align*}
To estimate these terms, we use Cauchy-Schwarz inequality. Using Theorem \ref{max-ineq} and the fact that $\sigma$ is bounded,
\begin{align*}
& \bE\left[\left|
\int_0^t \int_{K} e^{-i \xi \cdot y} \frac{\sin((t+h-s)|\xi|)-\sin((t-s)|\xi|)}{|\xi|} \sigma(u_N(s,y)) L^M(ds,dy) \right|^4 \right]\leq \\
& \quad \quad C \left\{ \left(\int_0^t \frac{|\sin((t+h-s)|\xi|)-\sin((t-s)|\xi|)|^2}{|\xi|^2} ds\right)^2+ \right.\\
& \quad \quad \quad \quad \quad \left. \int_0^t \frac{|\sin((t+h-s)|\xi|)-\sin((t-s)|\xi|)|^4}{|\xi|^4} ds \right\}\leq C h^4
\end{align*}
and
\begin{align*}
& \bE\left[\left|
\int_t^{t+h} \int_{K} e^{-i \xi \cdot y} \frac{\sin((t+h-s)|\xi|))}{|\xi|} \sigma(u_N(s,y)) L^M(ds,dy) \right|^4 \right]\leq \\
& \quad \quad C \left\{ \left(\int_t^{t+h} \frac{\sin^2((t+h-s)|\xi|)}{|\xi|^2} ds\right)^2+\int_t^{t+h} \frac{\sin^4((t+h-s)|\xi|)}{|\xi|^4} ds \right\}\leq C h^5.
\end{align*}
Similar bounds are valid for the other two terms which involve $t-h$. This proves that \eqref{mom-incr} holds with $\delta=3$. So, $\{u_{N}^{1,1}(t,\cdot)\}_{t \in [0,T]}$ has a c\`adl\`ag modification with values in $H^r(\bR^2)$.

\medskip

\underline{We treat $u_N^2(t,x)$.} If $J$ has points $(T_i,X_i,Z_i)$ in $[0,T] \times \bR^2 \times \bR$, then
\[
u_N^2(t,x)=\sum_{i\geq 1}G_{t-T_i}(x-X_i) \sigma(u_N(T_i,X_i))Z_i 1_{\{T_i \leq t\}} 1_{\{1<|Z_i| \leq N h(X_i)\}}.
\]
Let $\varphi \in C_c^{\infty}(\bR^2)$ be arbitrary. Pick $A>0$ such that ${\rm supp} \ \varphi \subset \{x\in \bR^2; |x| \leq A\}$,
Since $G_{t-T_i}(x-X_i)$ contains the indicator of the set $|x-X_i|<t-T_i$, we have $|X_i|<|x|+t \leq A+T$, for any $|x|\leq A$. Hence, for any $x \in \bR^2$,
\[
u_N^2(t,x)\varphi(x)=\sum_{i\geq 1}G_{t-T_i}(x-X_i) \varphi(x) \sigma(u_N(T_i,X_i))Z_i 1_{\{T_i \leq t\}} 1_{\{|X_i|<A+T,1<|Z_i| \leq N h(X_i)\}}.
\]
This sum contains finitely many terms,
since $S=\{(x,z);|x| <A+T, 1<|z|<Nh(x)\}$ is a bounded set in $\bR^2 \times \bR$. We look at one of these terms: for any $i\geq 1$ fixed,
the function $[0,T] \ni t \mapsto G_{t-T_i}(\cdot-X_i) \varphi$ is c\`adl\`ag in $H^r(\bR^2)$, by Lemma \ref{lem-G-t0}.

We claim that $[0,T]\ni t \mapsto u_N^2(t,\cdot)$ is c\`adl\`ag in $H_{\rm loc}^r(\bR^2)$. Right-continuity is clear, since for any $t \in [0,T]$ and for any $\varphi \in C_c^{\infty}(\bR^2)$,
\[
\lim_{h \to 0+}u_N^2(t+h,\cdot) \varphi=u_N^2(t,\cdot) \varphi \quad \mbox{in} \quad H^r(\bR^2).
\]
For the existence of left limit, we have to prove that for any $t\in [0,T]$, there exists $\phi(t) \in H_{\rm loc}^r(\bR^2)$ such that
\[
\lim_{h \to 0+}u_N^2(t-h,\cdot) \varphi=\phi(t) \varphi \quad \mbox{in} \quad H^r(\bR^2).
\]
To prove this, let $\varphi \in C_c^{\infty}(\bR^2)$ be arbitrary. Pick $A$ such that ${\rm supp} \ \varphi \subset \{x; |x| \leq A\}$. By Lemma \ref{lem-G-t0}, for any $i\geq 1$, there exists $\phi_i(t)\in H^r(\bR^2)$ such that
$\lim_{h \to 0+}G_{t-h-T_i}(\cdot -X_i) =\phi_i(t)$ in $H^r(\bR^2)$, and hence in $H_{\rm loc}^r(\bR^2)$ (since the embedding of $H^r(\bR^2)$ into $H_{\rm loc}^r(\bR^2)$ is continuous). It follows that $\lim_{h \to 0+}G_{t-h-T_i}(\cdot -X_i) \varphi =\phi_i(t)\varphi$ in $H^r(\bR^2)$ for any $i\geq 1$, and
\begin{align*}
u_N^2(t-h,\cdot) \varphi&=\sum_{i\geq 1}G_{t-h-T_i}(\cdot-X_i) \varphi \sigma(u_N(T_i,X_i))Z_i 1_{\{T_i \leq t\}} 1_{\{|X_i|<A+T,1<|Z_i| \leq N h(X_i)\}}
\end{align*}
converges in $H^r(\bR^2)$ as $h \to 0+$ to
\[
\sum_{i\geq 1}\phi_i(t) \varphi \sigma(u_N(T_i,X_i))Z_i 1_{\{T_i \leq t\}} 1_{\{|X_i|<A+T,1<|Z_i| \leq N h(X_i)\}}=:\phi(t) \varphi.
\]

\medskip

\underline{We treat $u_N^3(t,x)$.}
By Lemma 2.13 of \cite{CDH19},
the function $t \mapsto u^3(t,\cdot)$ is continuous in $H_{\rm loc}^r(\bR^2)$, since $(t,x) \mapsto G_t(x)$ is integrable on $[0,T] \times \bR^2$ and $\sigma$ is bounded.

\medskip

{\em Case 2. $p\in (0,1)$.} Using the non-drift decomposition \eqref{no-drift-LN} of $L_N$, we write
\begin{align}
\nonumber
u_N(t,x)&=w(t,x)+\int_0^t \int_{\bR^2}G_{t-s}(x-y) \sigma(u_N(s,y))L^Q(ds,dy)+\\
\nonumber
& \quad \quad \quad \quad \quad \int_0^t \int_{\bR^2}G_{t-s}(x-y) \sigma(u_N(s,y))L_N^P(ds,dy)\\
\label{decomp-uN-1}
&=:w(t,x)+u_N^1(t,x)+u_N^2(t,x).
\end{align}
The term $u_N^2$ is the same as in Case 1.

To treat $u_N^1$, we use the same argument as in Case 1, with $L^M$ replaced by $L^Q$. We mention only the changes, and we omit the details. Let $A\geq T$ be arbitrary. Using the compact set $K=\{y\in \bR^2;|y|\leq 2A\}$, we write $u_N^1(t,x)=u_N^{1,1}(t,x)+u_N^{1,2}(t,x)$ with
\begin{align}
\label{def-uN-11-1}
u_N^{1,1}(t,x)&=\int_0^t \int_{K} G_{t-s}(x-y)\sigma(u_N(s,y))L^Q(ds,dy)\\
\label{def-uN-12-1}
u_N^{1,2}(t,x)&=\int_0^t \int_{K^c} G_{t-s}(x-y)\sigma(u_N(s,y))L^Q(ds,dy),
\end{align}
The same argument as above shows that $u_N^{1,2}(t,\cdot)1_{\{|x|\leq A\}}=0$.

It remains to study $u_N^{1,1}$. In this case, the Fourier transform of $u_N^{1,1}(t,\cdot)$ is given by
\begin{equation}
\label{Fourier-uN-11-1}
a_{\xi}(t):=\cF u_N^{1,1}(t,\cdot)(\xi)=\int_0^t \int_{K}e^{-i\xi \cdot y}\frac{\sin((t-s)|\xi|)}{|\xi|}\sigma(u_N(s,y))L^Q(ds,dy),
\end{equation}
This follows by applying twice the stochastic Fubini theorem given by Theorem A.3.2 of \cite{CDH19}. For the first application of this theorem, we need to check that
\[
\int_0^t \int_{K}\int_{\{|z|\leq 1\}}
\bE\left[ \left(\int_{\bR^2}G_{t-s}(x-y)|\sigma(u_N(s,y))\cF \varphi(x)|dx\right)^p \right]|z|^p \nu(dz)dyds<\infty.
\]
This follows using Assumption A, \eqref{p-mom-finite}, $|\cF \varphi(x)|\leq \|\varphi\|_{L^1(\bR^2)}$ and $\int_{\bR^2}G_{t-s}(x-y)dx=t-s$. For the second application of Fubini's theorem, we need to check that
\[
\int_0^t \int_{K}\int_{\{|z|\leq 1\}}
\bE\left[ \left(\int_{\bR^2}\left|e^{-i \xi \cdot y} \frac{\sin((t-s)|\xi|)}{|\xi|}\sigma(u_N(s,y)) \varphi(\xi)\right|d\xi\right)^p \right]|z|^p \nu(dz)dyds<\infty.
\]
This follows using Assumption A, \eqref{p-mom-finite}, and the fact that
\[
\int_{\bR^2} \frac{|\sin((t-s)|\xi|)|}{|\xi|}|\varphi(\xi)|d \xi \leq \|\varphi\|_{L^2(\bR^2)}\left( \int_{\bR^2} \frac{\sin^2((t-s)|\xi|)}{|\xi|^2}d \xi\right)^{1/2} \leq C,
\]
since $\sin^2((t-s)|\xi|)/|\xi|^2$ can be bounded by $(t-s)^2$ if $|\xi|\leq 1$ and by $1/|\xi|^2$ if $|\xi|>1$.

For any $t \in [0,T]$, $u_N^{1,1}(t,\cdot)\in H^r(\bR^2)$ with probability $1$, because
\begin{align*}
&\bE\left[\int_{\bR^2}|a_{\xi}(t)|^2(1+|\xi|^2)^{r}d\xi\right] \\
&=\int_{\bR^2} \bE \left|\int_0^t \int_{K} \int_{|z|\leq 1}e^{-i \xi \cdot y} \frac{\sin((t-s)|\xi|)}{|\xi|}
\sigma(u_N(s,y))z J(ds,dy,dz) \right|^2 (1+|\xi|^2)^{r}d\xi\\
&\leq C \int_{\bR^2} \left\{\int_0^t \int_{K} \int_{|z|\leq 1}\frac{\sin^2((t-s)|\xi|)}{|\xi|^2}
\bE|\sigma(u_N(s,y))|^2 |z|^2 \nu(dz) dyds
+\right.\\
&\quad \quad \quad \left.\bE\left[\left(\int_0^t \int_{K} \int_{|z|\leq 1}\frac{|\sin((t-s)|\xi|)|}{|\xi|}
|\sigma(u_N(s,y))| |z| \nu(dz) dyds\right)^2 \right]\right\}
(1+|\xi|^2)^{r}d\xi \\
& \leq C \int_{\bR^2} \left\{\int_0^t \frac{\sin^2((t-s)|\xi|)}{|\xi|^2} ds+\left(\int_0^t \frac{|\sin((t-s)|\xi|)|}{|\xi|} ds\right)^2\right\}(1+|\xi|^2)^{r}d\xi<\infty,
\end{align*}
using Lemma \ref{mom-N}, the fact that $\sigma$ is bounded, $\int_{|z|\leq 1}|z|\nu(dz) <\infty$
and $r<-1$.

The fact that $\{u_{N}^{1,1}(t,\cdot)\}_{t \in [0,T]}$ is stochastically continuous in $H^r(\bR^2)$ follows from
\[
\bE|a_{\xi}(t+h)-a_{\xi}(t)|^2 \leq C h^2,
\]
which is proved using an inequality similar to \eqref{a-incr-1} with $L^M$ replaced by $L^Q$, followed by an application of Lemma \ref{mom-N}. Finally, relation \eqref{mom-incr} also holds with $\delta=3$. To prove this, we proceed as in Case 1, replacing $L^M$ by $L^Q$, and applying Lemma \ref{mom-N} with $p=4$.
\end{proof}

We consider now the case of an unbounded function $\sigma$.

\begin{theorem}
Suppose that $d=2$ and Assumption A holds with $p<2$. Let $\{u(t,x);t\in [0,T],x \in \bR^2\}$ be the solution to equation \eqref{wave-eq} on the interval $[0,T]$, constructed in Theorem \ref{exist-th} but with stopping times $(\tau_N)_{N\geq 1}$ defined by \eqref{def-tau2}. Then the process $\{u(t,\cdot)\}_{t \in [0,T]}$ has a c\`adl\`ag modification with values in $H_{\rm loc}^r(\bR^2)$, for any $r<-1$.
\end{theorem}

\begin{proof}
We use the same argument as in the proof of Theorem 2.15 of \cite{CDH19} for the heat equation. It suffices to prove the result for $u_N$, for arbitrary $N\geq 1$. We consider separately the cases $p\geq 1$ and $p \in (0,1)$.

\medskip

{\em Case 1. $p \geq 1$.} We use decomposition \eqref{decomp-uN}.

\medskip

\underline{We treat $u_N^1(t,x)$.}
We write $u_{N}^{1}(t,x)=u_{N}^{1,1}(t,x)+u_{N}^{1,2}(t,x)$ with $u_{N}^{1,1}(t,x)$ and $u_{N}^{1,2}(t,x)$ given by \eqref{def-uN-11}, respectively \eqref{def-uN-12}. Then $u_{N}^{1,2}(t,\cdot)1_{\{|x|\leq A\}}=0$ for any $A \geq T$.

It remains to treat $u_{N}^{1,1}(t,x)$.
The Fourier transform of $u_N^{1,1}(t,\cdot)$ is still given by \eqref{Fourier-uN-11} (the above justification of this formula does not use the boundedness of $\sigma$). But if $\sigma$ is unbounded, it is not immediately clear why $\int_{\bR^2}|a_{\xi}(t)|^2(1+|\xi|^2)^rd\xi<\infty$ a.s.

Consider the truncated function $\sigma_n(u)=\sigma(u)1_{\{|u|\leq n\}}$, and define
\[
u_{N,n}^{1,1}(t,x)=\int_0^t \int_{K} G_{t-s}(x-y) \sigma_n(u_{N}(s,y)) L^M(ds,dy).
\]

Since $\sigma_n$ is bounded,
$\{u_{N,n}^{1,1}(t,\cdot)\}_{t \in [0,T]}$ has a c\`adl\`ag modification in $H^r(\bR^2)$, for any $n\geq 1$.
This follows from the proof of Theorem \ref{th-sigma-bded}.
Since the uniform limit of a sequence of c\`adl\`ag functions is c\`adl\`ag, it suffices to show that with probability $1$, the sequence $\{u_{N,n}^{1,1}(t,\cdot)\}_{n\geq 1}$ converges to $u_{N}^{1,1}(t,\cdot)$ as $n \to \infty$, uniformly in $t \in [0,T]$.
To achieve this, we will use a change of measure, as in \cite{CDH19}.
Let $\sigma_{(n),N}(s,y)=\sigma(u_N(s,y)) 1_{\{|u_N(s,y)|>n\}}$. Then
\[
u_{N}^{1,1}(t,x)-u_{N,n}^{1,1}(t,x)=\int_0^t \int_{K}G_{t-s}(x-y) \sigma_{(n),N}(s,y)L^M(ds,dy)
\]
\[
a_{\xi}^{(n)}(t):=\cF \big(u_{N}^{1,1}(t,\cdot)-u_{N,n}^{1,1}(t,\cdot)\big)(\xi)=\int_0^t \int_{K} e^{-i \xi \cdot y} \frac{\sin((t-s)|\xi|)}{|\xi|}
\sigma_{(n),N}(s,y)  L^M(ds,dy)
\]
\begin{equation}
\label{diff-uN}
\sup_{t \in [0,T]} \|u_{N}^{1,1}(t,\cdot)-u_{N,n}^{1,1}(t,\cdot)\|_{H^r(\bR^2)}^2  \leq \int_{\bR^2} (1+|\xi|^2)^r \sup_{t \in [0,T]}   |a_{\xi}^{(n)}(t)|^2 d\xi.
\end{equation}
To evaluate $a_{\xi}^{(n)}(t)$, we use the fact that
$\frac{\sin((t-s)|\xi|)}{|\xi|}=\int_s^t \cos((t-r)|\xi|)dr$. By the stochastic Fubini theorem given by Theorem A.3.1 of \cite{CDH19},
\[
a_{\xi}^{(n)}(t)=\int_0^t \cos((t-r)|\xi|) \left(\int_0^r \int_K
e^{-i \xi \cdot y}\sigma_{(n),N}(s,y)  L^M(ds,dy) \right) dr.
\]
To justify the application of this theorem, we need to check that:
\[
\int_0^t \left(\int_0^r \int_{K} \int_{|z|\leq 1}
\bE|e^{-i \xi \cdot y} \cos((t-r)|\xi|) \sigma_{(n),N}(s,y) z|^p \nu(dz) dy ds
 \right)^{1/p} dr<\infty.
\]
This follows by Assumption A, \eqref{p-mom-finite} and the bound
$|\cos((t-r)|\xi|)| \leq 1$. For any $t \in [0,T]$,
\[
|a_{\xi}^{(n)}(t)|  \leq \sup_{r \in [0,T]}\left| \int_0^r \int_{K}e^{-i \xi \cdot y} \sigma_{(n),N}(s,y) L^M(ds,dy)\right| \cdot \int_0^t |\cos((t-r)|\xi|)|dr,
\]
and hence
\begin{equation}
\label{bound-a}
\sup_{t \in [0,T]}|a_{\xi}^{(n)}(t)|
\leq T \sup_{r \in [0,T]}\left| \int_0^r \int_{K}e^{-i \xi \cdot y} \sigma_{(n),N}(s,y) L^M(ds,dy)\right|.
\end{equation}

We intend to apply Theorem A.4 of \cite{CDH19} to the random measure $L^M$. First, we check that $\sigma(u_N) 1_K \in L^{1,p}(L^M)$ (we refer to Appendix A of \cite{CDH19} for the notation). To see this, we use  relation (A.3) of \cite{CDH19}, Assumption A and \eqref{p-mom-finite}:
\[
\|\sigma(u_N) 1_K \|_{L^M,p}^p \leq C \int_0^T \int_K \int_{\{|z| \leq 1\}} \bE|\sigma(u_N(s,y))|^p |z|^p \nu(dz) dyds<\infty.
\]
By Theorem A.4 of \cite{CDH19}, there exists a probability measure $\bQ$ on $(\Omega,\cF)$ which is equivalent to $\bP$, such that $\frac{d\bQ}{d\bP}$ is bounded, $\frac{d\bP}{d\bQ} \in L^{\frac{2}{2-p}}(\Omega,\cF,\bP)$, $L^M$ is an $L^2$-random measure under $\bQ$ and $\sigma(u_N)1_K \in L^{1,2}(L^M,\bQ)$.
Hence,
\[
\bE_{\bQ}\left[\sup_{t \in [0,T]}\left|\int_0^t \int_{K} e^{-i \xi \cdot y} \sigma_{(n),N}(s,y) L^M(ds,dy) \right|^2\right]
\leq \|\sigma_{(n),N}1_{K}\|_{L^M,2,\bQ}^2,
\]
where $\bE_{\bQ}$ denotes the expectation with respect to $\bQ$.

We return to \eqref{diff-uN}. We use inequality \eqref{bound-a}, then we take expectation with respect to $\bQ$. We obtain that
\begin{align*}
& \bE_{\bQ}\left[\sup_{t \in [0,T]}\|u_{N}^{1,1}(t,\cdot)-u_{N,n}^{1,1}(t,\cdot)\|_{H^r(\bR^2)}^2\right]
\\
& \quad \leq T^2 \int_{\bR^2}(1+|\xi|^2)^r \bE_{\bQ}\left[\sup_{t \in [0,T]}\left|\int_0^t \int_{K} e^{-i \xi \cdot y} \sigma_{(n),N}(s,y) L^M(ds,dy) \right|^2\right] d\xi \\
& \quad \leq T^2 \|\sigma_{(n),N}1_{K}\|_{L^M,2,\bQ}^2 \int_{\bR^2}(1+|\xi|^2)^r d\xi \to 0 \quad \mbox{as} \quad n\to \infty,
\end{align*}
using the (stochastic) dominated convergence theorem given by Theorem A.1 of \cite{CDH19}, applied to the random measure $L^M$ (under the probability measure $\bQ$). 
In our case, the integrand $\sigma_{(n),N}1_{K} \to 0$ converges pointwise to $0$ as $n \to \infty$, and is bounded by the random function $|\sigma(u_N)1_K|$ which belongs to $L^{1,2}(L^M,\bQ)$.

Therefore, there exists a subsequence $N'\subset \bN$ such that with probability $1$,
\[
\sup_{t \in [0,T]}\|u_{N}^{1,1}(t,\cdot)-u_{N,n}^{1,1}(t,\cdot)\|_{H^r(\bR^2)} \to 0 \quad \mbox{as} \quad n \to \infty,n\in N'.
\]
Since $\{u_{N,n}^{1,1}(t,\cdot)\}_{t \in [0,T]}$ has a c\`adl\`ag modification in $H^r(\bR^2)$ for any $n\geq 1$, the process
$\{u_{N}^{1,1}(t,\cdot)\}_{t \in [0,T]}$ will inherit this property too.

\medskip

\underline{We treat $u_N^2(t,x)$.} The same argument as in the proof of Theorem \ref{th-sigma-bded} shows that $\{u_N^2(t,\cdot)\}_{t \in [0,T]}$ is c\`adl\`ag in
$H_{\rm loc}^r(\bR^2)$.

\medskip

\underline{We treat $u_N^3(t,x)$.} Let $A\geq T$ be arbitrary and $K=\{y \in \bR^2;|y| \leq 2A\}$. We write $u^{3}(t,x)=u^{3,1}(t,x)+u^{3,2}(t,x)$, where
\begin{align}
\label{def-uN-31}
u^{3,1}(t,x) &= b\int_0^t \int_{K}G_{t-s}(x-y) \sigma(u_N(s,y))dyds \\
\label{def-uN-32}
u^{3,2}(t,x) &= b\int_0^t \int_{K^c}G_{t-s}(x-y) \sigma(u_N(s,y))dyds.
\end{align}
As for $u^{1,2}$, we see that $u^{3,2}(t,\cdot)1_{\{|x|\leq A\}}=0$, using the compact support property of $G$.

It remains to treat $u^{3,1}$. For this, we consider again the truncated function $\sigma_n$ as above and we let
\[
u_{N,n}^{3,1}(t,x)=b\int_0^t \int_{K}G_{t-s}(x-y) \sigma_n(u_N(s,y))dyds.
\]
Since $\sigma_n$ is bounded, $\{u_{N,n}^{3,1}(t,\cdot)\}_{t \in [0,T]}$ is continuous in $H_{\rm loc}^r(\bR^2)$, by Lemma 2.13 of \cite{CDH19}.

Similarly to \eqref{bound-a} (but with $ds dy$ instead of $L^M(ds,dy)$), we have
\begin{align*}
\sup_{t \in [0,T]}|\cF \big(u_{N}^{3,1}(t,\cdot)-u_{N,n}^{3,1}(t,\cdot)\big)(\xi)| & \leq bT \sup_{r \in [0,T]}\left|\int_0^r \int_{K} e^{-i \xi \cdot y} \sigma_{(n),N}(s,y) dyds \right| \\
&\leq
 bT \int_0^T \int_{K}|\sigma_{(n),N}(s,y)|dyds,
\end{align*}
and hence
\[
\sup_{t \in [0,T]}\|u_N^{3,1}(t,\cdot)-u_{N,n}^{3,1}(t,\cdot)\|_{H^r(\bR^2)}^2 \to 0 \quad \mbox{as} \quad n \to\infty.
\]
It follows that $\{u_{N}^{3,1}(t,\cdot)\}_{t \in [0,T]}$ is continuous in $H_{\rm loc}^r(\bR^2)$.

\medskip

{\em Case 2. $p \in (0,1)$.} We use \eqref{decomp-uN-1}. The term $u_N^2$ is the same as in Case 1. To treat $u_N^1$, let $A \geq T$ be arbitrary and $K=\{y \in \bR^2;|y|\leq 2A\}$. We write $u_N^1=u_N^{1,1}+u_{N}^{1,2}$, where $u_N^{1,1}$ and $u_{N}^{1,2}$ are given by \eqref{def-uN-11-1}, respectively \eqref{def-uN-12-1}. Then $u_N^{1,2}(t,\cdot)1_{\{|x|\leq A\}}=0$.

It remains to study $u_N^{1,1}$. The Fourier transform of $u_{N}^{1,1}(t,\cdot)$ is given by
\eqref{Fourier-uN-11-1}. The justification of this formula does not use the fact that $\sigma$ is bounded. Let $\sigma_n(u)=\sigma(u) 1_{\{|u|\leq n\}}$ and define
\[
u_{N,n}^{1,1}(t,x)=\int_0^t \int_{K}G_{t-s}(x-y)\sigma_n(u_{N}(s,y))L^Q(ds,dy).
\]
Since $\sigma_n$ is bounded, $\{u_{N,n}^{1,1}(t,\cdot)\}_{t \in [0,T]}$ has a c\`adl\`ag modification in $H^r(\bR^2)$. This follows from the proof of Theorem \ref{th-sigma-bded} in the case $p<1$. As in Case 1, replacing $L^M$ by $L^Q$, we have:
\begin{align*}
a_{\xi}^{(n)}(t)&:=\cF \big(u_{N}^{1,1}(t,\cdot)-u_{N,n}^{1,1}(t,\cdot)\big)(\xi)=\int_0^t \int_{K} e^{-i \xi \cdot y} \frac{\sin((t-s)|\xi|)}{|\xi|}
\sigma_{(n),N}(s,y)  L^Q(ds,dy) \\
&=\int_0^t \left(\int_0^r \int_{K} e^{-i \xi \cdot y} \cos((t-r)|\xi|)\sigma_{(n),N}(s,y)  L^Q(ds,dy) \right) dr.
\end{align*}
The last equality above is due to the stochastic Fubini theorem given by Theorem A.3.2 of \cite{CDH19}. To justify the application of this theorem, we need to check that
\[
\int_0^t \int_{K} \int_{\{|z|\leq 1\}} \bE \left[\left( \int_s^t |e^{-i \xi \cdot y} \cos((t-r)|\xi|) \sigma_{(n),N}(s,y)| dr \right)^p\right] |z|^p \nu(dz) dy ds<\infty.
\]
This follows using Assumption A, \eqref{p-mom-finite} and the bound $|\cos((t-r)|\xi|)|\leq 1$.

We now apply the change of measure of Theorem A.4 of \cite{CDH19} to the random measure $L^Q$. To check that $\sigma(u_N) 1_K \in L^{1,p}(L^Q)$, we use Lemma A.2.3 of \cite{CDH19}, Assumption A and \eqref{p-mom-finite}:
\[
\|\sigma(u_N) 1_K \|_{L^Q,p} \leq \int_0^T \int_K \int_{\{|z| \leq 1\}} \bE|\sigma(u_N(s,y))|^p |z|^p \nu(dz) dyds<\infty.
\]
The rest of the proof is the same as in Case 1.
\end{proof}

\subsection{Case $d=1$}

In this case, $G_t \in H^r(\bR)$ for any $r<1/2$ and $t>0$. On the other hand, $G_0=\frac{1}{2}1_{\{0\}}$ since
\[
G_0(x):=\lim_{t\to 0+}G_t(x)=
\left\{
\begin{array}{ll}
0 & \mbox{if $x \not =0$} \\
1/2 & \mbox{if $x=0$}
\end{array} \right.
\]
Therefore, $G_0 \in H^r(\bR)$ for any $r \in \bR$. For $t>0$, $G_t \in  H^r(\bR)$ if $r<1/2$, due to \eqref{FG-bound}.

\medskip

The following result is the counterpart of Lemma \ref{lem-G-t0} for $d=1$.

\begin{lemma}
\label{lem-G-t0-d1}
If $d=1$, for any $t_0 \in [0,T]$, $x_0 \in \bR^2$, the map $[0,T] \ni t\mapsto G_{t-t_0}(\, \cdot-x_0)$ is c\`adl\`ag in $H^r(\bR^2)$ for any $r<1/2$.
\end{lemma}

\begin{proof}
We have to prove that the function $F:[0,T] \to H^r(\bR^2)$ given by
\[
F(t)=
\left\{
\begin{array}{ll}
G_{t-t_0}(\, \cdot -x_0) & \mbox{if $t >t_0$} \\
2^{-1}1_{\{x_0\}} & \mbox{if $t=t_0$}\\
0 & \mbox{if $t<t_0$}
\end{array} \right.
\]
is c\`adl\`ag. The fact that $F$ is continuous at any point $t>t_0$ follows as in the proof of Lemma \ref{lem-G-t0-d1}.
$F$ is right-continuous at $t_0$ since
\begin{align*}
\|F(t_0+h)-F(t_0)\|_{H_r^2(\bR)}&=\int_{\bR} |\cF G_h(\cdot-x_0)(\xi)-\frac{1}{2}\cF 1_{\{x_0\}}(\xi)|^2 (1+|\xi|^2)^r d\xi\\
&=\int_{\bR^2}\left|e^{-i \xi \cdot x_0} \frac{\sin(h|\xi|)}{|\xi|} \right|^2 (1+|\xi|^2)^r d\xi \to 0
\end{align*}
as $h \to 0+$, by the dominated convergence theorem.
\end{proof}

Unlike the case $d=2$, the solution $u_N$ of the equation with truncated noise has finite moments of order $p$, for any $p\geq 2$. Therefore, we do not need to consider separately the case of a bounded function $\sigma$. Note that we impose a new restriction $r<1/4$ (compared with Lemma \ref{lem-G-t0-d1}) which comes from the analysis of increments; this introduces additional requirements on the initial conditions $u_0$ and $v_0$.

\begin{theorem}
Assume that $d=1$ and there exists $q>0$ such that
\[
\int_{\{|z|>1\}}|z|^q \nu(dz)<\infty.
\]
In addition to the initial conditions mentioned at the beginning, we suppose that $u_0,v_0 \in L^1(\bR)$ and $|\cF u_0(\xi)|\leq c \frac{1}{1+|\xi|^2}$ for any $\xi \in \bR$ for some constant $c>0$.

Let $\{u(t,x);t\in [0,T],x \in \bR\}$ be the solution to equation \eqref{wave-eq} on the interval $[0,T]$, constructed in Theorem \ref{exist-th} but with stopping times $(\tau_N)_{N\geq 1}$ defined by \eqref{def-tau2}. Then the process $\{u(t,\cdot)\}_{t \in [0,T]}$ has a c\`adl\`ag modification with values in $H_{\rm loc}^r(\bR^2)$, for any $r<1/4$.
\end{theorem}

\begin{proof}
Since $\int_{|z|>1}|z|^{q}\nu(dz)$ is nondecreasing in $q$, we can assume that $q\leq 2$. Therefore, Assumption A holds with $p\geq 2$ arbitrary and $q\in (0,2]$. By Theorem \ref{exist-th-K},
\begin{equation}
\label{finite-2-mom}
\sup_{t \in [0,T]} \sup_{|x|\leq R}\bE |u_{N}(t,x)|^{p}<\infty \quad \mbox{for any} \quad p\geq 2,
\end{equation}
for any $T>0$ and $R>0$. As mentioned above, it is enough to prove the result for $u_N$. Since $p>1$ in Assumption A, we use decomposition \eqref{decomp-uN}.

\medskip

\underline{We treat $w(t,x)$.}
We have
\[
w(t,x)=\frac{1}{2}\int_{x-t}^{x+t}v_0(y)dy+\frac{1}{2}[u_0(x+t)-u_0(x-t)]=:
w^1(t,x)+w^2(t,x).
\]
We treat separately the two terms. Note that $\cF w^1(t,\cdot)(\xi)=\cF G_t(\xi)\cF v_0(\xi)$ and
\begin{align*}
& \|w(t+h,\cdot)-w(t,\cdot)\|_{H^r(\bR)}^2=\int_{\bR}\frac{|\sin((t+h)|\xi|)-\sin(t|\xi|)|^2}{|\xi|^2} |\cF v_0(\xi)|^2 (1+|\xi|^2)^r d\xi\\
&\quad \quad \quad =\int_{\bR}\frac{4}{|\xi|^2} \sin^2\left(\frac{h|\xi|}{2} \right)\cos^2\left(\frac{2t+h}{2}|\xi|\right) |\cF v_0(\xi)|^2 (1+|\xi|^2)^r d\xi \to 0 \quad \mbox{as} \quad h\to0,
\end{align*}
by the dominated convergence theorem, using the fact that $|\cF v_0(\xi)| \leq \|v_0\|_{L^1(\bR)}$ and
$\frac{1}{|\xi|^2} \sin^2\left(\frac{h|\xi|}{2}\right) \leq C \frac{1}{1+|\xi|^2}$ for any $\xi \in \bR$. To apply this theorem, we only need $r<1/2$.

By direct calculation, $\cF w^2(t,\cdot)(\xi)=i\sin(\xi t) \cF u_0(\xi)$. Hence,
\begin{align*}
& \|w^2(t+h,\cdot)-w^2(t,\cdot)\|_{H^r(\bR)}^2 =\int_{\bR}
|\sin(\xi (t+h))-\sin(\xi t)|^2 |\cF u_0(\xi)|^2 (1+|\xi|^2)^r d\xi \\
& \quad \quad \int_{\bR} 4 \sin^2 \left(\frac{h|\xi|}{2} \right) \cos^2 \left( \frac{(2t+h)|\xi|}{2}\right) |\cF u_0(\xi)|^2 (1+|\xi|^2)^r d\xi
\to 0 \quad \mbox{as} \quad h\to0,
\end{align*}
by the dominated convergence theorem, using the fact that $|\cF u_0(\xi)|^2 \leq c \frac{1}{1+|\xi|^2}$ for all $\xi \in \bR$, and $r<1/2$.

\medskip

\underline{We treat $u_N^1(t,x)$.} We use the decomposition $u_N=u_N^1+u_N^2$ where $u_N^1,u_N^2$ are given by \eqref{def-uN-11}, respectively \eqref{def-uN-12}, and $K=[-2A,2A]$.
As in the case $d=2$, $u_{N}^{1,2}(t,\cdot)1_{[-A,A]}=0$ for any $A\geq T$. So it suffices to consider $u_N^{1,1}$.

The same argument as in the proof of Theorem \ref{th-sigma-bded} shows that $\{u_N^{1,1}(t,\cdot)\}_{t\in [0,T]}$ is stochastically continuous as $H^r(\bR^2)$-valued process, for any $r<1/2$. Instead of using the boundedness of $\sigma$, we now use \eqref{finite-2-mom} and the fact that $\sigma$ is Lipschitz.

Next, we prove \eqref{mom-incr}. For this, we use \eqref{bound-E} and \eqref{bound-E2}. By the Cauchy-Schwarz inequality,
\[
E_{t,h}^1(\xi,\eta)\leq I_1^{1/2}I_2^{1/2}, \quad E_{t,h}^2(\xi,\eta)\leq I_1^{1/2}I_3^{1/2}, \quad E_{t,h}^3(\xi,\eta)\leq I_4^{1/2}I_2^{1/2}, \quad E_{t,h}^4(\xi,\eta)\leq I_4^{1/2}I_3^{1/2},
\] where
\begin{align*}
I_1 &=\bE\left[ \left|\int_{0}^t \int_K e^{-i \xi y}
\frac{\sin((t+h-s)|\xi|-\sin((t-s)|\xi|))}{|\xi|}\sigma(u_N(s,y)) L^M(ds,dy)\right|^4 \right] \\
I_2 &=\bE\left[ \left|\int_{0}^t \int_K e^{-i \xi y}
\frac{\sin((t-h-s)|\eta|-\sin((t-s)|\eta|))}{|\eta|}\sigma(u_N(s,y)) L^M(ds,dy)\right|^4 \right]\\
I_3 &=\bE\left[ \left|\int_{t-h}^t \int_K e^{-i \xi y}
\frac{\sin((t-h-s)|\eta|)}{|\eta|}\sigma(u_N(s,y)) L^M(ds,dy)\right|^4 \right] \\
I_4 &=\bE\left[ \left|\int_{t}^{t+h} \int_K e^{-i \xi y}
\frac{\sin((t+h-s)|\xi|}{|\xi|}\sigma(u_N(s,y)) L^M(ds,dy)\right|^4 \right].
\end{align*}

We apply Theorem \ref{max-ineq}, followed by Cauchy-Schwarz inequality and \eqref{finite-2-mom}:
\begin{align*}
I_1 & \leq C\bE\left[ \left(\int_{0}^t \int_K
\frac{|\sin((t+h-s)|\xi|)-\sin((t-s)|\xi|))|^2}{|\xi|^2}|\sigma(u_N(s,y))|^2 dy ds\right)^2 \right]+\\
& \quad \quad C\bE\left[\int_{0}^t \int_K
\frac{|\sin((t+h-s)|\xi|)-\sin((t-s)|\xi|))|^4}{|\xi|^4}|\sigma(u_N(s,y))|^4 dy ds  \right] \\
& \leq  C\int_{0}^t \int_K
\frac{|\sin((t+h-s)|\xi|)-\sin((t-s)|\xi|))|^4}{|\xi|^4}\bE|\sigma(u_N(s,y))|^4 dy ds  \\
& \leq C \int_{0}^t
\frac{|\sin((t+h-s)|\xi|)-\sin((t-s)|\xi|))|^4}{|\xi|^4} ds\\
& \leq C \frac{1}{|\xi|^4}\left|\sin\left(\frac{h|\xi|}{2}\right)\right|^4=C \frac{1}{|\xi|^{\e}}\left|\sin\left(\frac{h|\xi|}{2}\right)\right|^{\e}\cdot
\frac{1}{|\xi|^{4-\e}}\left|\sin\left(\frac{h|\xi|}{2}\right)
\right|^{4-\e}\\
&\leq C h^{\e}\left(\frac{1}{1+|\xi|^2}\right)^{\frac{4-\e}{2}} \quad \mbox{for any} \quad \e \in [0,4]
\end{align*}
and
\begin{align*}
I_4 & \leq C \bE\left[\left(\int_{t}^{t+h} \int_K
\frac{|\sin((t+h-s)|\xi|)|^2}{|\xi|^2}|\sigma(u_N(s,y))|^2 dy ds\right)^2 \right]+\\
& \quad \quad C\bE\left[\int_{t}^{t+h} \int_K
\frac{|\sin((t+h-s)|\xi|)|^4}{|\xi|^4}|\sigma(u_N(s,y))|^4 dy ds  \right] \\
& \leq  C\int_{t}^{t+h} \int_K
\frac{|\sin((t+h-s)|\xi|)|^4}{|\xi|^4}\bE|\sigma(u_N(s,y))|^4 dy ds  \\
& \leq C \int_{t}^{t+h}
\frac{|\sin((t+h-s)|\xi|)|^4}{|\xi|^4} ds=C \int_{0}^h
\frac{|\sin(s|\xi|)|^4}{|\xi|^4} ds\\
&=C \int_{0}^h
\frac{|\sin(s|\xi|)|^{\e}}{|\xi|^{\e}} \cdot \frac{|\sin(s|\xi|)|^{4-\e}}{|\xi|^{4-\e}}ds \leq C \left(\frac{1}{1+|\xi|^2}\right)^{\frac{4-\e}{2}} \int_0^h s^{\e}ds\\
&= C h^{1+\e}\left(\frac{1}{1+|\xi|^2}\right)^{\frac{4-\e}{2}} \quad \mbox{for any} \quad \e \in [0,4].
\end{align*}
Similarly, for any $\e \in [0,4]$,
\[
I_2 \leq  C h^{\e}\left(\frac{1}{1+|\eta|^2}\right)^{\frac{4-\e}{2}} \quad \mbox{and} \quad I_3 \leq C h^{1+\e}\left(\frac{1}{1+|\eta|^2}\right)^{\frac{4-\e}{2}}.
\]
It follows that
\[
E_{t,h}(\xi,\eta)\leq C h^{\e}\left(\frac{1}{1+|\xi|^2}\right)^{\frac{4-\e}{4}}
\left(\frac{1}{1+|\eta|^2}\right)^{\frac{4-\e}{4}}
\]
and
\[
\bE[\|u(t+h,\cdot)-u(t,\cdot)\|_{H^r(\bR)}^2 \cdot
\|u(t-h,\cdot)-u(t,\cdot)\|_{H^r(\bR)}^2 ]\leq C h^{\e} \left[ \int_{\bR}
\left(\frac{1}{1+|\xi|^2}\right)^{\frac{4-\e}{4}-r}d\xi\right]^2.
\]
The last integral converges if and only if $\e<2-4r$. On the other hand, we need to choose $\e>1$. This introduces the restriction $r<1/4$. This proves \eqref{mom-incr}.
By Theorems 1 and 5 of \cite{GS74}, the process $\{u_N^1(t,\cdot)\}_{t \in [0,T]}$ has a c\'adl\'ag modification in $H^r(\bR)$ for any $r<1/4$.

\medskip

\underline{We treat $u_N^2(t,x)$.} Using the same argument as in the proof of Theorem \ref{th-sigma-bded} and Lemma \ref{lem-G-t0-d1}, we see that the process $\{u_N^2(t,\cdot)\}_{t \in [0,T]}$ is c\`adl\`ag in $H^r(\bR)$ for any $r<1/2$.

\medskip

\underline{We treat $u_N^3(t,x)$.} Let $A>0$ be arbitrary and $K=[-2A,2A]$. We write $u_N^3(t,x)=u_N^{3,1}(t,x)+u_N^{3,2}(t,x)$, where $u_N^{3,1}$ and $u_N^{3,2}$ are given by \eqref{def-uN-31} and \eqref{def-uN-32}.
Using the compact support property of $G$, we see that
$u_N^{3,2}(t,\cdot)1_{[-A,A]}=0$ for any $A>T$.

So, it suffices to consider $u_{N}^{3,1}$. Note that the Fourier transform of $u_{N}^{3,1}(t,\cdot)$ is given by
\[
\cF u_N^{3,1}(t,\cdot)(\xi)=b \int_0^t \int_{K}e^{-i \xi y} \frac{\sin((t-s)|\xi|)}{|\xi|} \sigma(u_{N}(s,y))dyds.
\]
We consider separately the cases when $\sigma$ is bounded and unbounded.

\medskip

{\em Case 1.} ($\sigma$ is bounded) By Minkowski's inequality,
\begin{align*}
\|u_{N}^{3,1}(t,\cdot)\|_{H^r(\bR)}&=\|\cF u_{N}^{3,1}(t,\cdot)\|_{L^2(\bR,(1+|\xi|^2)^r d\xi)} \\
& \leq b \int_0^t \int_{K} \left(\int_{\bR} \frac{\sin^2((t-s)|\xi|)}{|\xi|^2} (1+|\xi|^2)^r d\xi \right)^{1/2} |\sigma(u_N(s,y))| dyds\\
&\leq C \left[\int_{\bR} \left(\frac{1}{1+|\xi|^2} \right)^{1-r}d\xi \right]^{1/2}<\infty \quad \mbox{since} \quad r<1/2.
\end{align*}
We claim that $\{u_{N}^{3,1}(t,\cdot)\}_{t \in [0,T]}$ is continuous in $H^r(\bR)$, for any $r<1/2$. To see this, we write
\begin{align*}
&\cF \big(u_{N}^{3,1}(t+h,\cdot)-u_{N}^{3,1}(t,\cdot)\big)(\xi)=b \int_t^{t+h} \int_{K}\frac{\sin((t+h-s)|\xi|)}{|\xi|} \sigma(u_N(s,y))dyds+\\
& \quad \quad \quad \quad \quad \quad b\int_0^t \int_{K}e^{-i \xi y} \frac{\sin((t+h-s)|\xi|)-\sin((t-s)|\xi|)}{|\xi|}\sigma(u_N(s,y))dyds.
\end{align*}
Then applying Minkowski's inequality as above, we obtain:
\begin{align*}
& \|u_{N}^{3,1}(t+h,\cdot)-u_{N}^{3,1}(t,\cdot)\|_{H^r(\bR)} \leq\\
& \quad b \int_t^{t+h} \int_{K} \left(\int_{\bR} \frac{\sin^2((t+h-s)|\xi|)}{|\xi|^2} (1+|\xi|^2)^r d\xi \right)^{1/2} |\sigma(u_N(s,y))| dyds+\\
& \quad b \int_0^t \int_{K} \left(\int_{\bR} \frac{|\sin((t+h-s)|\xi|)-\sin((t-s)|\xi|)|^2}{|\xi|^2} (1+|\xi|^2)^r d\xi \right)^{1/2} |\sigma(u_N(s,y))| dyds
\end{align*}
and the last two integrals converge to $0$ as $h\to 0$, by the dominated convergence theorem.

\medskip

{\em Case 2.} ($\sigma$ is general) For any $n\geq 1$, let $\sigma_n(x)=\sigma(x) 1_{\{|x|\leq n\}}$ and define
\[
u_{N,n}^{3,1}(t,x)=b \int_0^t \int_{K}G_{t-s}(x-y)\sigma_n(u_{N}(s,y))dyds.
\]

By Case 1 above, $\{u_{N,n}^{3,1}(t,\cdot)\}_{t \in [0,T]}$ is continuous in $H^r(\bR)$, for any $n\geq 1$ and $r<1/2$. We fix $r<1/2$. We will prove that along a subsequence, with probability $1$, $\{u_{N,n}^{3,1}\}_{n\geq 1}$ converges to $u_{N}^{3,1}$ in $H^r(\bR)$ as $n \to \infty$, uniformly in $t \in [0,T]$. Since the uniform limit of continuous functions is continuous,  $\{u_{N}^{3,1}(t,\cdot)\}_{t \in [0,T]}$ will be continuous in $H^r(\bR)$.

We denote $\sigma_{(n),N}(s,y)=\sigma(u_N(s,y))-\sigma_n(u_N(s,y))$. By Fubini theorem,
\begin{align*}
\cF \big(u_{N}^{3,1}(t,\cdot)-u_{N,n}^{3,1}(t,\cdot) \big)(\xi)&=b\int_0^t \int_{K} e^{-i \xi y} \frac{\sin((t-s)|\xi|)}{|\xi|} \sigma_{(n),N}(s,y) dyds \\
&=b \int_0^t \cos((t-r)|\xi|) \left(\int_0^r \int_{K} e^{-i \xi y} \sigma_{(n),N}(s,y) dyds \right) dr.
\end{align*}
Hence,
\begin{align*}
\sup_{t \in [0,T]}|\cF \big(u_{N}^{3,1}(t,\cdot)-u_{N,n}^{3,1}(t,\cdot) \big)(\xi)|
& \leq b \left(\int_0^T \int_{K}|\sigma_{(n),N}(s,y)|dyds\right)
\left(\int_0^T |\cos(r|\xi|)|dr\right).
\end{align*}

We claim that
\[
\int_0^T |\cos(r|\xi|)|dr \leq C \left( \frac{1}{1+|\xi|^2}\right)^{1/2}.
\]
To see this, assume that $T \in \left((\frac{2k-1}{2}\vee 0)\pi,\frac{2k+1}{2}\pi\right)$ for some integer $k\geq 0$.
Say $k$ is even, $k=2m$ for some integer $m\geq 0$.
(The case when $k$ is odd is similar.) Then
\[
\int_0^T |\cos(r|\xi|)|dr =\frac{\sin(\frac{\pi}{2}|\xi|)}{|\xi|}+\sum_{\ell=1}^{m-1}\frac{
\sin(\frac{4\ell+1}{2}\pi)-\sin(\frac{4\ell-1}{2}\pi)}{|\xi|}-
\sum_{\ell=1}^{m} \frac{
\sin(\frac{4\ell-1}{2}\pi)-\sin(\frac{4\ell-3}{2}\pi)}{|\xi|},
\]
and we use the fact that for any $a<b$,
\[
\frac{|\sin(a|\xi|)-\sin(b|\xi|)|}{|\xi|}\leq 2\frac{|\sin\left(\frac{a-b}{2}|\xi|\right)|}{|\xi|}
\leq 2 \left\{2\left[\left(\frac{a-b}{2}\right)^2 \vee 1\right]\frac{1}{1+|\xi|^2}\right\}^{1/2}.
\]

It follows that
\begin{align*}
\bE\Big[\sup_{t \in [0,T]}|\cF \big(u_{N}^{3,1}(t,\cdot)-u_{N,n}^{3,1}(t,\cdot) \big)(\xi)|^2\Big] & \leq C \bE\left[\frac{1}{1+|\xi|^2}\left(\int_0^T \int_{K}|\sigma_{(n),N}(s,y)|dyds\right)^2\right] \\
& \leq C \frac{1}{1+|\xi|^2}\int_0^T \int_{K}\bE|\sigma_{(n),N}(s,y)|^2dyds
\end{align*}
and
\begin{align*}
& \bE\Big[\sup_{t \in [0,T]}\|u_{N}^{3,1}(t,\cdot)- u_{N,n}^{3,1}(t,\cdot)\|_{H^r(\bR)}^2 \Big]   \\
& \quad \quad \quad \leq C
\left(\int_{\bR} \frac{1}{1+|\xi|^2} (1+|\xi|^2)^r d\xi \right)
\left(\int_0^T \int_{K} \bE|\sigma_{(n),N}(s,y)|^2
dyds\right).
\end{align*}
The last integral converges to $0$ as $h \to 0$, by the dominated convergence theorem. Hence, there exists a subsequence $N' \subset \bN$ such that $\sup_{t \in [0,T]}\|u_{N}^{3,1}(t,\cdot)- u_{N,n}^{3,1}(t,\cdot)\|_{H^r(\bR)}^2\to 0$ a.s. when $n\to\infty$, $n \in N'$.

\end{proof}

\appendix

\section{Stochastic integral}

For the reader's convenience, we here include the definition of the stochastic integral with respect to an $L^p$-random measure. We refer the reader to
\cite{chong17-JTP,chong17-SPA,CDH19} for more details.

Let $\widetilde{\Omega}=\Omega \times \bR_{+} \times \bR^d$ and $\cP=\cP_0 \times \cB(\bR^d)$, where $\cP_0$ is the predictable $\sigma$-field on $\Omega \times \bR_+$.
Let $p\in [0,\infty)$ be arbitrary. Given a sequence $(\widetilde{\Omega}_k)_{k\geq 1}\subset \cP$ satisfying
$\widetilde{\Omega}_k \uparrow \widetilde{\Omega}$, a map $M:\cP_M=\bigcup_{k\geq 1} \cP|_{\widetilde{\Omega}_k} \to L^p(\Omega)$ is called an {\em $L^p$-random measure} if for any disjoint sets $(A_i)_{i\geq 1}$ in $\cP_{M}$ such that $\bigcup_{i\geq 1}A_i \in \cP_M$, we have $M(\bigcup_{i\geq 1}A_i)=\sum_{i\geq 1}M(A_i)$ in $L^p(\Omega)$, and some additional adaptedness conditions hold.

If $S$ is a {\em simple integrand} of the form $S=\sum_{i=1}^{k} a_i 1_{A_i}$ for some $a_i \in \bR$ and $A_i \in \cP_M$, then the stochastic integral of $S$ with respect to $M$ is given by:
\[
I^M(S)=\int_{0}^{\infty}\int_{\bR^d}S(t,x)M(dt,dx)=\sum_{i=1}^{k}a_i M(A_i).
\]
Let $\cS_M$ be the set of all simple integrands. The {\em Daniell mean} of a process $H=\{H(t,x);t\geq 0,x \in \bR^d\}$ with respect to $M$ is defined by
\[
\|H\|_{M,p}=\sup_{S \in \cS_{M};|S|\leq |H|} \|I^M(S)\|_p.
\]
A predictable process $H$ is said to be {\em $p$-integrable with respect to $M$} if there exists a sequence $(S_n)_{n\geq 1} \subset \cS_M$ such that $\|S_n-H\|_{M,p} \to 0$ as $n \to \infty$. In this case, the {\em stochastic integral} of $H$ with respect to $M$ is defined by:
\[
I^M(S)=\lim_{n\to \infty}I^M(S_n) \quad \mbox{in} \quad L^p(\Omega).
\]
We denote by $L^{1,p}(M)$ the set of $p$-integrable processes with respect to $M$. Then the map
$I^M:L^{1,p}(M) \to L^p(\Omega)$ is a contraction.
In these definitions, we omit writing $p$, if $p=0$.

\section{Moment Inequalities}

In this section, we include the moment inequalities which were used in the sequel. The first result is a Rosenthal-type maximal inequality (see also Theorem 1 of \cite{marinelli-rockner14}).

\begin{theorem}[Theorem 2.3 of \cite{BN16}]
\label{max-ineq}
For any predictable process $H$ and for any $p\geq 2$,
\begin{align*}
&\bE\left[\sup_{t\leq T} \left|\int_0^t \int_{\bR^d} \int_{\bR}H(s,x,z)\widetilde{J}(ds,dx,dz)\right|^p\right]\leq \\
& \quad C_p\left\{ \bE \left[\left(\int_0^T \int_{\bR^d}\int_{\bR}H^2(t,x,z)\nu(dz)dxdt \right)^{p/2}\right]+\bE \left[\int_0^T \int_{\bR^d}\int_{\bR}|H(t,x,z)|^p \nu(dz)dx dt \right]\right\},
\end{align*}
where $C_p>0$ is a constant depending on $p$.
\end{theorem}

To control the $p$-th moments of integrals with respect to $J$, we use the following result.

\begin{lemma}
\label{mom-N}
For any predictable process $H$ such that
$\int_0^T \int_{\bR^d}\int_{\bR}|H(t,x,z)|\nu(dz)dxdt<\infty$ a.s.
and for any $p\geq 2$,
\begin{align*}
& \bE\left|\int_0^T \int_{\bR^d}\int_{\bR}H(t,x,z)J(dt,dx,dz)\right|^p  \leq \\
& \quad C_p\left\{ \bE \left[\left(\int_0^T \int_{\bR^d}\int_{\bR}H^2(t,x,z)\nu(dz)dxdt \right)^{p/2}\right]+\bE \left[\int_0^T \int_{\bR^d}\int_{\bR}|H(t,x,z)|^p \nu(dz)dx dt \right] \right. \\
& \quad \quad + \left.\bE \left[\left(\int_0^T \int_{\bR^d}\int_{\bR}|H(t,x,z)|\nu(dz)dxdt \right)^{p}\right]\right\},
\end{align*}
where $C_p>0$ is a constant depending on $p$.
\end{lemma}

\begin{proof}
This follows from Theorem \ref{max-ineq}, writing
$\int H dJ=\int H d \widetilde{J} +\int H dt dx \nu(dz)$.
\end{proof}

The next result considers the case $p\leq 2$.

\begin{theorem}
\label{max-ineq-th}
a) For any predictable process $H$ and for any $p\in [1,2]$,
\begin{align*}
\bE\left[\sup_{t\leq T} \left|\int_0^t \int_{\bR^d} \int_{\bR}H(s,x,z)\widetilde{J}(ds,dx,dz)\right|^p\right]\leq C_p\bE \left[\int_0^T \int_{\bR^d}\int_{\bR}|H(t,x,z)|^p \nu(dz)dx dt \right],
\end{align*}
where $C_p>0$ is a constant depending on $p$.

b) For any predictable process $H$ and for any $p\in (0,1)$,
\begin{align*}
\bE\left|\int_0^t \int_{\bR^d} \int_{\bR}H(s,x,z)J(ds,dx,dz)\right|^p \leq \bE \left[\int_0^T \int_{\bR^d}\int_{\bR}|H(t,x,z)|^p \nu(dz)dx dt \right].
\end{align*}
\end{theorem}

\begin{proof} a) The process $M_t=\int_0^t \int_{\bR^d} \int_{\bR}H(s,x,z)\widetilde{J}(ds,dx,dz)$ is a martingale with quadratic variation:
\[
[M]_t=\int_0^t \int_{\bR^d} \int_{\bR}H^2(s,x,z)J(ds,dx,dz).
\]
By the Burkholder-Davis-Gundy inequality, since $p\geq 1$, $\bE(\sup_{t\leq T}|M_t|^p)\leq C_p \bE([M]_t^{p/2})$. Since $p/2 \leq 1$, it can be proved that:
$$
[M]_t^{p/2} \leq \int_0^t \int_{\bR^d} \int_{\bR}|H(s,x,z)|^pJ(ds,dx,dz)
$$
(see e.g. the proof of Lemma 8.22 of \cite{PZ07}). The conclusion follows taking expectation.

b) This follows using the inequality $|x+y|^p \leq |x|^p+|y|^p$. We refer to the proof of Lemma A.2.(3) of \cite{CDH19}.
\end{proof}

\vspace{3mm}

\noindent \footnotesize{{\em Acknowledgement.} The author is grateful to Carsten Chong and Gennady Samorodnitsky for useful discussions.

\normalsize{

\end{document}